\documentclass[a4paper,reqno,11pt]{amsart}

\usepackage{etoolbox}
\patchcmd{\section}{\scshape}{\bfseries\large}{}{}
\patchcmd{\subsubsection}{\itshape}{}{}{}
\makeatletter
\def\@seccntformat#1{\csname the#1\endcsname.\space}
\makeatother

\usepackage{amsmath, amsfonts, amssymb, amsthm, amscd}
\usepackage{graphicx}
\usepackage{psfrag}
\usepackage{perpage}
\usepackage{url}
\usepackage{color}
\usepackage{mathrsfs}
\usepackage{mathabx}
\usepackage{scalerel}
\usepackage{enumerate}
\usepackage{comment}
\usepackage[normalem]{ulem}
\usepackage{tikz}\usepackage{caption,subcaption}
\usetikzlibrary{arrows,decorations.pathmorphing,backgrounds,positioning,fit,petri}
\usepackage{tikz-cd} 
\usepackage{subdepth}

\usepackage{dsfont} 

\usepackage[utf8]{inputenc}
\usepackage[T1]{fontenc}
\usepackage{microtype}

\usepackage[a4paper,scale={0.72,0.74},marginratio={1:1},footskip=7mm,headsep=10mm]{geometry}

\usepackage{hyperref}

\allowdisplaybreaks[0] 

\numberwithin{equation}{section}

\newtheorem{theorem}{Theorem}[section]
\newtheorem{lemma}[theorem]{Lemma}
\newtheorem{proposition}[theorem]{Proposition}

\newtheorem{conjecture}[theorem]{Conjecture}

\theoremstyle{definition}
\newtheorem{definition}[theorem]{Definition}
\newtheorem{question}{Question}
\newtheorem{remark}[theorem]{Remark}

\newcommand{\E}{\mathbb{E}}

\renewcommand{\P}{\mathbb{P}}

\newcommand{\R}{\mathbb{R}}
\newcommand{\Z}{\mathbb{Z}}
\newcommand{\T}{\mathbb{T}}

\def\bs{\boldsymbol}

\newcommand\bP{\ensuremath{\bs{\mathrm{P}}}}
\newcommand\bE{\ensuremath{\bs{\mathrm{E}}}}

\newcommand{\cC}{{\ensuremath{\mathcal C}} }

\newcommand{\cT}{{\ensuremath{\mathcal T}} }

\newcommand{\be}{\begin{equation}}
\newcommand{\ee}{\end{equation}}

\definecolor{darkorange}{RGB}{255, 100, 0}

\newcommand{\ClCp}[2]{ \cT_{\cC_{#1}(#2)}} 
\newcommand{\Conep}{\mathcal{C}_{1}(p)}
\newcommand{\bigO}[1]{O\Big(#1\Big)}

\newcommand{\e}{{\rm e}}
\newcommand{\weight}{\textup{\textrm{w}}}

\newcommand{\eps}{\varepsilon}
\renewcommand{\eps}{\varepsilon}
\renewcommand{\theta}{\vartheta}
\renewcommand{\rho}{\varrho}

\newcommand{\effR}[3]{{\rm R}_{\rm eff}^{#1}(#2 \leftrightarrow #3)} 
\newcommand{\tmix}{\ensuremath{t_{\text{\normalfont mix}}}}

\makeatletter
\newcommand\xleftrightarrow[2][]{%
  \ext@arrow 9999{\longleftrightarrowfill@}{#1}{#2}}
\newcommand\longleftrightarrowfill@{%
  \arrowfill@\leftarrow\relbar\rightarrow}
\makeatother

\usepackage{marginnote}
\newcounter{Chapcounter}

\newcommand{\chapter}[1]
{ {\centering
  \addtocounter{Chapcounter}{1} \Large \underline{\textbf{ \color{blue} Chapter \theChapcounter: ~#1}} }
  \addcontentsline{toc}{section}{ \color{blue} Chapter:~\theChapcounter~~ #1}
}

\MakePerPage[2]{footnote} 

\date{\today}

\title[Random Spanning Trees in Random Environment]{Random Spanning Trees in Random Environment}

\begin{document}

\author[L. Makowiec]{Luca Makowiec}
\address{Department of Mathematics\\
National University of Singapore\\
10 Lower Kent Ridge Road, 119076 Singapore
}
\email{makowiec@u.nus.edu}

\author[M. Salvi]{Michele Salvi}
\address{
Department of Mathematics\\
University of Rome Tor Vergata\\
Via della Ricerca Scientifica 1, 00133 Rome, Italy
}
\email{salvi@mat.uniroma2.it}

\author[R. Sun]{Rongfeng Sun}
\address{Department of Mathematics\\
National University of Singapore\\
10 Lower Kent Ridge Road, 119076 Singapore
}
\email{matsr@nus.edu.sg}


\begin{abstract}
We introduce a new spanning tree model called the \textit{random spanning tree in random environment} (RSTRE), which interpolates between the uniform spanning tree and the minimum spanning tree as the inverse temperature (disorder strength) $\beta$ varies. On the complete graph with $n$ vertices and i.i.d.\ uniform disorder variables on the edges, we identify: (1) a low disorder regime with $\beta \leq C n/\log n$, where the diameter of the random spanning tree is typically of order $n^{1/2}$, the same as for the uniform spanning tree; (2) a high disorder regime with $\beta \geq n^{4/3} \log n$, where the diameter is typically of order $n^{1/3}$, the same as for the minimum spanning tree. We conjecture that for $\beta=n^{\alpha}$ with $\alpha \in (1, 4/3)$, the diameter is of  order $n^{\gamma+o(1)}$ for some $\gamma=\gamma(\alpha)$ strictly between $1/2$ and $1/3$.
\end{abstract}

\maketitle

\setcounter{tocdepth}{2}
\makeatletter
\def\l@subsection{\@tocline{2}{0pt}{2.5pc}{5pc}{}}
\makeatother

\tableofcontents



\section{Introduction} \label{S:Intro}


Let $G = (V,E)$ be a connected finite graph with vertex set $V$ and edge set $E$. Let $\omega:=(\omega_e)_{e \in E}$ be i.i.d.\ random variables assigned to the edges with common distribution $\mu$, which constitute the disorder (random environment). Probability and expectation with respect to $\omega$ will be denoted by $\mathbb{P}$ and $\E$, respectively.
A {\em spanning tree} $\cT$ of $G$ is a cycle-free connected subgraph of $G$ with the same vertex set $V$. We identify $\cT$ with its own edge set and write $\T_G$ for the set of all spanning trees on $G$. 

We introduce here a new model of random spanning trees on $G$ defined as follows. Given a realization of the random environment $\omega$, we define for each $\beta\geq 0$ a Gibbs measure on $\T_G$ by letting
\begin{equation}\label{eq:PomegaT}
	\bP^\omega_{G, \beta}(\cT) 
	:= \frac{1}{Z_{G, \beta}^{\omega}} \prod_{e \in \cT} \e^{-\beta \omega_e} = \frac{1}{Z_{G, \beta}^{\omega}} \, \e^{-\beta H(\cT, \omega)}\,,
\end{equation}
where the Hamiltonian is defined as $H(\cT, \omega):= \sum_{e\in \cT}\omega_e$ and the normalization constant $Z_{G, \beta}^{\omega}$, also called partition function, is given by
\begin{equation}\label{eq:Zomega}
	Z_{G, \beta}^{\omega} :=  \sum_{\cT \in \T_G} \e^{-\beta H(\cT, \omega)}\,.
\end{equation}
We call $\bP^\omega_{G, \beta}$ the law of the \textit{random spanning tree in random environment} (RSTRE),  with expectation denoted by $\bs{\mathrm{E}}^\omega_{G, \beta}$. This law may equivalently be interpreted as the {\em uniform spanning tree} (UST) measure on the weighted graph $(G,\weight)$ with edge weights (or conductances)
\begin{equation}\label{eq:we}
	\weight(e) := \weight_\beta(e):= \exp(-\beta \omega_e).
\end{equation} 
Unless specified otherwise, we will reserve the term UST for the unweighted graph (with $\weight(e)\equiv 1$), and the term weighted UST for the UST on the weighted graph $(G, \weight)$. For more background material on the UST, see e.g.\ \cite[Chapter 4]{LP16}. 

We note that when $\beta=0$, $\bP^\omega_{G, 0}$ is simply the UST measure on the unweighted graph $G$. When $\mu$ is a continuous distribution,
the measure $\bP^\omega_{G, \infty} :=\lim_{\beta\uparrow\infty} \bP^\omega_{G, \beta}$ concentrates on the a.s.\ unique $\cT \in \T_G$ that minimizes
the Hamiltonian $H(\cT, \omega)$. This tree is known as the {\em minimum spanning tree} (MST) on $G$, see e.g.\ \cite[Chapter 11]{LP16}, and it only depends on the ordering of $(\omega_e)_{e\in E}$.  As $\beta$ increases from $0$ to $\infty$, $\bP^\omega_{G, \beta}$ interpolates between the UST and MST on $G$: our main goal is to understand how this transition occurs.

\smallskip

In this work, we will mainly investigate the RSTRE on the complete graph $G=K_n = (V_n, E_n)$ with $n$ vertices and i.i.d.~disorder variables $\omega_e$ on the edges with uniform distribution $\mu$ on $[0,1]$.  
This is motivated by the fact that there is a fairly good understanding of the UST and MST on the complete graph,
see e.g.\ \cite{Ald91a, Ald91b, Ald93, PR05, Sch09, ABGM17}. We denote by $\cT^\omega_{n, \beta}$ the random tree sampled according to $\bP^{\omega}_{n, \beta_n} := \bP_{K_n, \beta}^{\omega}$ in the random environment $\omega$. When 
there is no ambiguity, we will simply denote $\cT^\omega_{n, \beta}$ by $\cT$. On $K_n$ with $n$ large, it is known that the UST has typically a diameter of order $n^{1/2}$, while the MST has a diameter of order $n^{1/3}$. As we will show later, to see the transition from the UST to MST measured through the typical diameter of $\cT=\cT^\omega_{n, \beta}$, we need to choose $\beta=\beta_n$ that grows with the number of vertices $n$.

\subsection{Main results}
Since many properties hold only asymptotically, we will need to recall some standard notation. Namely, for two functions $f,g$ we say that $f = O(g)$ (respectively $\Omega(g)$) if there exists some constants $C > 0$ and $n_0$ such that for all $n \geq n_0$ it holds that $f(n) \leq C g(n)$ (respectively $f(n) \geq C g(n)$). If both $f = O(g)$ and $g = O(f)$, then we write $f \asymp g$ or $f = \Theta(g)$. Furthermore, we will write $f \ll g$ (or $f = o(g)$) if $\lim_{n \rightarrow \infty} f(n)/g(n) = 0$. Lastly, we say that a sequence of properties $\mathcal{A}_n$ holds w.h.p.\ (with high probability) if their probability goes to $1$ as $n\to\infty$.

We will denote by $\widehat \P$ (with expectation $\widehat \E$) the joint law of $\omega$ and $\cT$ where $\omega$ has marginal law $\P$, and conditional on $\omega$, the  random spanning tree $\cT$ has law $\bP^\omega_{n, \beta_n}$. The marginal law 
$$
\widehat \P(\cT \in \cdot) = \E[ \bP^\omega_{n, \beta_n}(\cT\in \cdot)]
$$ 
is known as the \textit{averaged law}, while $\bP^\omega_{n, \beta_n}$ is known as the \textit{quenched law}.

\smallskip

Recall that the diameter of a connected graph $G=(V,E)$ is the maximal graph distance $d_G$ between any pair of its vertices, where the graph distance between two vertices is the minimal number of edges one has to cross to go from one to the other:
\begin{equation*}
	{\rm diam}(G):=\max\limits_{u,v \in V} d_G(u,v)\,.
\end{equation*}

Our main result is the following theorem, which identifies a low (resp.\ high) disorder regime of $\beta=\beta_n$, such that under the averaged law $ \widehat{\P}$, the diameter of $\cT^\omega_{n, \beta_n}$ is of the same polynomial order as the UST (resp.\ MST).

\begin{theorem} \label{T:main}
	For $\beta_n\geq 0$, let $\cT^\omega_{n, \beta_n}$ be the RSTRE on the complete graph $K_n=(V_n, E_n)$ with $n$ vertices and i.i.d.\ disorder variables $(\omega_e)_{e\in E_n}$ uniformly distributed on $[0,1]$.
	There exists a constant $C>0$ such that for any $\delta > 0$, if $\beta_n \leq C n/\log n$ and $n$ is sufficiently large, then 
	\begin{equation}\label{eq:low}
		\widehat \P \Big( C_1(\delta)^{-1} n^{1/2}\leq{\rm diam}(\cT^\omega_{n, \beta_n})\leq C_1(\delta) n^{1/2} 
		\Big) \geq 1-\delta\,,
	\end{equation}
	where $C_1(\delta)>0$ is a constant  depending only on $\delta$.  
	On the other hand, if $\beta_n \geq n^{4/3} \log n$ and $n$ is sufficiently large, then 	
	\begin{equation}\label{eq:high}
		\widehat \P \Big( C_2(\delta)^{-1} n^{1/3}\leq{\rm diam}(\cT^\omega_{n, \beta_n})\leq C_2(\delta) n^{1/3}   \Big) \geq 1-\delta\,,
	\end{equation}
	where $C_2(\delta)>0$ is a constant  depending only on $\delta$.  
\end{theorem}

We call $\beta_n \leq C n/\log n$ the {\em low disorder regime} and $\beta_n \geq n^{4/3} \log n$ the {\em high disorder
	regime}. When $\beta$ does not depend on $n$, the edge weights are all within constant multiples of each other and one can use the results of 
\cite{AS24} about UST on dense graphs, combined with standard concentration arguments, to show that the properly rescaled $\cT^\omega_{n, \beta}$
converges in law to Aldous' Brownian \textit{continuum random tree} (CRT), see Proposition \ref{P:GHPbeta} below.

\smallskip

\subsection{Related literature} 
To the best of our knowledge, this is the first time the RSTRE model appears in the literature.\footnote{Shortly after we posted our paper, K\'{u}sz posted a paper \cite{K24} introducing the same model with essentially the same result as our Theorem \ref{T:main}. See the end of Section \ref{S:Intro} for more details.} Our definition of the RSTRE is inspired by the {\em directed polymer
	model in random environment} (DPRE), defined as follows. Given i.i.d.\ random variables $\omega(k, x)$ indexed by the time-space lattice points $(k, x)\in {\mathbb N}\times \Z^d$, 
the directed polymer of length $N$ at inverse temperature $\beta$ is defined via the Gibbs measure
$$
\bP^{\omega}_{N, \beta}(S) := \frac{1}{Z_{N, \beta}^{\omega}} \e^{ \beta \sum_{k=1}^N \omega(k, S_k)} \bP(S),
$$
where $S$ is a simple symmetric random walk on $\Z^d$ with law $\bP$, and $Z_{N, \beta}^\omega$ is the partition function. Similar to the RSTRE, as $\beta$
increases from $0$ to $\infty$, the DPRE interpolates between the simple random walk measure  $\bP_{N, 0}^{\omega}=\bP$ and the measure $\bP_{N, \infty}^{\omega}$, which concentrates on the random walk path $S$ that minimizes the Hamiltonian $H^\omega_N(S) = -\sum_{k=1}^N \omega(k, S_k)$. 
The latter model is known as the {\em last passage percolation model} and belongs to the celebrated KPZ universality class. For more on the DPRE, see the monograph \cite{Com17}. 

\smallskip

Our motivation for introducing the RSTRE is to identify a natural interpolation between UST and MST. In particular, we wish to study the geometry of the random spanning
tree and its scaling limit. A first step in this direction is to study the diameter of the RSTRE. The following is known for the UST. For unweighted {\em ``high-dimensional''}   graphs, such as the complete graph, finite tori of dimension $d \geq 5$, expanders and dense graphs, the diameter of the UST is typically of order $n^{1/2}$, where $n$ is the number of vertices in the graph, see \cite{Sze83, Ald90, CHL12, PR05, MNS21, ANS22}. On these graphs, in fact, it is believed that, seen as a random metric space equipped with the graph distance, the UST rescaled by ${n^{-1/2}}$ would converge in distribution to Aldous' {\em Brownian continuum random tree} (CRT) \cite{Ald91a, Ald91b, Ald93}. This was verified in the Gromov-Hausdorff topology for the complete graph in \cite{Ald91a, Ald91b, Ald93}, in the sense of finite-dimensional distributions for finite tori of dimension $d\geq 5$ and $d=4$ in \cite{PR05} and \cite{Sch09} respectively, and in the stronger Gromov-Hausdorff-Prokhorov (GHP) topology for finite tori of dimension $d\geq 5$ in \cite{ANS22} and for dense graphs in \cite{AS24}. 
A precursor to the current work is \cite{MSS23}, where we essentially studied the RSTRE with arbitrary disorder distribution $\mu$ at a fixed $\beta$ that does not vary with the graph sequence $G_n$.
When the underlying graph is a bounded degree expander or a box of $\Z^d$ for $d \geq 5$, the diameter of the random spanning tree was shown in \cite{MSS23} to be of order $n^{1/2 + o(1)}$, falling in the ``class'' of the UST. 

The MST has also been studied extensively in the literature, see e.g.\ \cite{LPS06, ABGM17, NTW17} and the references therein. The MST is also connected to 
invasion percolation and the ground states of a spin glass model, see e.g.\ \cite{NS96}. Many basic questions about the MST remain open, such as for which dimensions does the MST on $\Z^d$ consist of more than one connected component in the infinite volume limit. Very little is known about the large scale geometries of MSTs and
their scaling limits. One exception is the MST on the complete graph, for which it was shown in \cite{ABGM17} that scaled by a factor of $n^{-1/3}$, the MST converges (as a random metric measure space) to a continuum limit that is different from Aldous' Brownian CRT. It is conjectured that the MST on ``high dimensional'' graphs should have the same scaling limit, which has been verified for the random 3-regular graph in \cite{AS21}. Another exception is \cite{GPS18}, where the authors show that a variant of the MST on the 2-dimensional triangular lattice converges to a continuum limit that is not conformally invariant. Furthermore, the scaling limits of the MST on the giant component of supercritical inhomogeneous random graph models have been studied in \cite{BS24}. To the best of our
knowledge, no other scaling limit results are known for the MST on $d$-dimensional lattices. 

\smallskip

Another model closely related to the RSTRE is first passage percolation with (random) edge weights $\weight_\beta(e):=\exp(\beta \omega_e)$. Given $\weight_\beta$,
the distance between any two vertices $u$ and $v$ is defined by minimising
\begin{equation} \label{eq:MSTweight}
	\sum_{e \in \Gamma} \weight_\beta(e)
\end{equation}
over all paths $\Gamma$ that connect $u$ and $v$ in the graph $G$. When $\beta=0$, this is exactly the graph distance on $G$. As $\beta\uparrow \infty$, the 
geodesics between vertices converge to paths in the MST on $G$ with edge variables $\omega_e$. See \cite{BvdH12, EGHN13,EGH20} for further details.

\subsection{Other results, conjectures and open questions}
Before stating some open questions, we first state a scaling limit result for the RSTRE on the complete graph $K_n$. The only regime where we can identify a non-trivial limit is when $\beta>0$ does not change with $n$. This will follow as a consequence of \cite[Theorem 1.1]{AS24}, and we will sketch the proof in Section \ref{SS:graphon}. 
\begin{proposition}\label{P:GHPbeta}
	Fix $\beta \geq 0$ and let $\cT^\omega_{n, \beta}$ be the RSTRE on the complete graph $K_n$ as in Theorem \ref{T:main}. Denote by $d_{\cT^\omega_{n, \beta}}$ the graph distance on $\cT^\omega_{n, \beta}$ and by $\nu_n$ the uniform distribution on the vertices of $K_n$. Then, for $\P$-almost all $\omega$, we have
	\begin{equation*}
		\big( \cT^\omega_{n, \beta}, \frac{1}{\sqrt{n}} d_{\cT^\omega_{n, \beta}}, \nu_n \big) \xrightarrow{(d)} \big( \cT, d_{\cT}, \nu \big),
	\end{equation*}
	where $(\cT, d_{\cT}, \nu)$ is Aldous' Brownian CRT, as defined in \cite{AS24}, equipped with its canonical mass measure $\nu$, and the convergence is in distribution with respect to the Gromov-Hausdorff-Prokhorov distance between metric measure spaces.
\end{proposition}

On the complete graph, Theorem \ref{T:main} leaves open what happens between the low
and high disorder regimes. In this intermediate regime, we make the following conjecture on the diameter of the random spanning tree.

\begin{conjecture} \label{C:Intermediate}
	Consider the same setting as in Theorem \ref{T:main}. Then, with high probability with respect to the averaged law $\widehat{\P}$, we have 
	\begin{equation*}
		{\rm diam} \big( \cT^\omega_{n, \beta_n} \big) =
		\begin{cases}
			n^{\frac{1}{2}+o(1)}, & \beta_n \leq n^{1 + o(1)}, \\
			n^{\frac{1}{2} - \frac{\gamma}{2}+o(1)}, & \beta_n  = n^{1 + \gamma+o(1)}, \quad 0 < \gamma < \frac{1}{3}, \\
			n^{\frac{1}{3}+o(1)}, & \beta_n \geq n^{\frac{4}{3} + o(1)}\,.
		\end{cases}
	\end{equation*}
\end{conjecture}

\noindent We refer to Section \ref{S:Conclusion} for a heuristic argument for the above conjecture. 

\smallskip

In a follow-up paper \cite{Mak24}, the first author studied local properties of the RSTRE under the same setup. In particular, if two trees $\cT$ and $\cT'$ (identified with their edge sets) are sampled independently with law $\bP^\omega_{n, \beta_n}$, then it is shown that, with high probability, the expected edge overlap 
$ \bE^{\omega, \otimes 2}_{n, \beta_n}[| \cT \cap \cT'|]$ 
grows as $\beta_n$ for $\beta_n \ll n/\log n$, while it is equal to $(1-o(1)) n$ for $\beta_n \gg n (\log n)^2$. Furthermore,  under the averaged law $\widehat\P$, the local limit of the RSTRE is equal to that of the UST (the Poisson($1$) tree conditioned on survival) for $\beta_n \ll n/\log n$, and equal to the local limit of the MST when $\beta_n\gg n(\log n)^c$ for all $c>0$. Hence, there is a sharp cutoff in the local limit behavior around $\beta_n = n$,  in contrast to the smooth transition in the diameter suggested by Conjecture \ref{C:Intermediate}.

\medskip

We list below a few more interesting open questions for the RSTRE. 

\begin{question}
	A more modest goal than proving Conjecture \ref{C:Intermediate} is to find a sequence $\beta_n$ such that, with high probability, 
	$$
	\frac{1}{3} < \liminf_{n\to\infty} \frac{1}{\log n} \log {\rm diam} \big( \cT^\omega_{n, \beta_n} \big) \leq \limsup_{n\to\infty} \frac{1}{ \log n} \log {\rm diam} \big( \cT^\omega_{n, \beta_n} \big) <\frac{1}{2}. 
	$$
\end{question}

\begin{question}
	What if we change the distribution $\mu$ of the disorder variables $\omega_e$? The more interesting cases would be when the law of
	$\weight(e)=\e^{-\beta\omega_e}$ is unbounded from above.
	For instance, let $\omega_e := -U^{-\gamma}_e$ for some $\gamma>0$, with $U_e$ uniformly distributed on $[0, 1]$ and fix $\beta\equiv1$. Is it true that ${\rm diam}(\cT^\omega_{n, 1})=n^{\alpha+o(1)}$ for some $\alpha=\alpha(\gamma) \in (1/3, 1/2)$? See also \cite[Remark~6.2]{MSS23}. 
\end{question}

\begin{question}
	Besides the diameter, are there other observables that can distinguish between the UST, RSTRE, and MST? One candidate is the expected edge overlap $ \bE^{\omega, \otimes 2}_{n, \beta_n}[| \cT \cap \cT'|]$ 
	between two random trees $\cT$ and $\cT'$ sampled independently according to the same Gibbs measure $\bP^\omega_{G, \beta}$ in \eqref{eq:PomegaT}. Can one discern the transition from UST to the MST (for any increasing sequence of graphs) by studying how the edge overlap varies with $\beta$? For the complete graph this was studied in \cite{Mak24}. 
\end{question}

\begin{question}
	Let $\cT$ be sampled according to the RSTRE measure $\bP^{\omega}_{n, \beta_n}$. For a typical vertex $v$, how does the volume of the metric ball $d_{\cT}(v,r)$ grow with the radius $r$? If ${\rm diam}(\cT)$ is of order $n^{\alpha+o(1)}$, then it is natural to conjecture that for $r \asymp {\rm diam}(\cT)$ we have that $d_{\cT}(v,r)$ is of order $r^{1/\alpha+o(1)}$, which suggests that any scaling limit of $\cT$ would have fractal dimension $1/\alpha$. For $r\ll {\rm diam}(\cT)$, this volume behavior might be different, see also \cite{Mak24}.
\end{question}

\begin{question}
	In Theorem \ref{T:main}, what is the scaling limit of the RSTRE in the low (resp.\ high) disorder regime? Is the scaling limit the same as that of the UST (resp.\ MST)?
\end{question}

\begin{question}
	Order all the possible spanning trees $T_1, T_2, \ldots$ of $K_n$ in increasing order of their Hamiltonian $H(T, \omega) = \sum_{e\in T} \omega_e$. Does the empirical distribution of $H(T_1, \omega), H(T_2, \omega), \ldots$ have scaling limits similar in spirit to the eigenvalues of a random matrix ensemble? How does ${\rm diam}(T_k)$ depend on $k$? For which choices of $k$ do we see a diameter of the order $n^\alpha$ with $\alpha\in (1/3, 1/2)$? What is the threshold $k_0=k_0(\beta_n)$ such that $\sum_{i=1}^{k_0} \bP^{\omega}_{n, \beta_n}(\cT=T_i)\geq 1/2$?
\end{question}

\begin{question}
	We could ask the same questions as above for the RSTRE on other graphs, in particular, for the finite torus on $\Z^d$. When $d\geq 5$, it is not difficult to identify a low disorder regime as in Theorem \ref{T:main}. Any understanding beyond such a low disorder regime could potentially help us understand better the MST on the torus. 
\end{question}

\subsection{Proof strategy and structure of the paper}
After recalling some basic tools and definitions in Section \ref{S:Notation},
we split the proof of Theorem \ref{T:main} into two parts, the low and high disorder regimes, which will be treated in Sections \ref{S:Low}  and \ref{S:High} respectively. The proof strategies are sketched below.  In Section \ref{S:Conclusion}, we will give a heuristic for Conjecture \ref{C:Intermediate}. 

\medskip

{\bf Low disorder regime:} 
In Section \ref{S:Low} we provide a more general result, see Theorem \ref{T:GeneralLow}, that implies the first part of Theorem \ref{T:main}. Indeed, Theorem \ref{T:GeneralLow} identifies a low disorder regime for any expander graph with bounded max-to-min degree ratio. The proof goes as follows.

For unweighted graphs $G=(V, E)$, general conditions have been formulated in \cite{MNS21} that imply the diameter of the UST on $G$ to be typically of order ${|V|}^{1/2}$. Roughly speaking, these conditions require that the stationary distribution of the random walk on $G$ assigns mass of the same order to all vertices, that the mixing time of the lazy random walk is of order $o(|V|^{1/2})$, and that the random walk return probabilities decay fast enough. 
In \cite{MSS23}, the authors have extended this result to weighted graphs and with some further generalizations: we recall this extension in Section \ref{SS:DiamterCondition}. We are then left to verify these conditions, which is achieved via concentration arguments for the stationary distribution and bottleneck ratio of the graph, see Section \ref{SS:expansion_concentrate}.
In Subsection \ref{SS:noconcentration} we explain why the concentration results fail as soon as $\beta_n\gg n / \log n$.

\smallskip

We conclude with Section \ref{SS:graphon}, where we prove Proposition \ref{P:GHPbeta}, namely, the convergence of the RSTRE on the complete graph to the CRT when $\beta$ does not depend on $n$. The proof follows from the results of \cite{AS24} once we show that the weighted graph $(K_n, \weight_n)$ converges to a (connected) positive constant graphon.

\medskip

{\bf High disorder regime:}  
Theorem \ref{T:highDisorder} in Section \ref{S:High}  restates the second part of Theorem \ref{T:main} and adds a result on the expectation of the diameter. Its proof strategy is adapted from \cite{ABR09}, where the authors deal with the diameter of the MST. The key differences with \cite{ABR09} are collected and explained in Remark \ref{R:MSTdiff}.

We begin the proof in Section \ref{SS:graph_reduce} by introducing a natural coupling between the random environment $\omega$ and a family of Erdős-R\'enyi random graphs, where edges are retained if the corresponding edge weight $\exp( {-\beta_n \omega_e})$ is large enough. This coupling is used in the crucial Proposition \ref{P:DiaC1}, where we show that the diameter of the random tree $\cT$ sampled according to the RSTRE is essentially determined by the diameter of the subset $\cT_{\mathcal C_1(p_0)}$ of $\cT$ that connects vertices in the largest component $\mathcal C_1(p_0)$ of an Erdős-R\'enyi random graph with edge retention probability $p_0$ that is critical, or slightly supercritical. The proof of Proposition \ref{P:DiaC1} is split into three parts: Lemma \ref{L:gap}, where we show that, with high probability, the vertices in the $p_0$-connected components of the random graph are connected in $\cT$ only using edges with high enough weight; Lemma \ref{L:sizepm}, where we show that the diameter of $\cT_{\mathcal C_1(p_0)}$ is comparable to that of the larger set $\cT_{\mathcal C_1(p_m)}$, where $p_m$ is the slightly supercritical parameter that makes the cardinality $|C_1(p_m)|$ of order $n/\log n$; finally  Lemma \ref{L:fasthitLERW} takes into account all vertices outside of $\mathcal C_1(p_m)$, showing that a loop erased random walk starting from those vertices hits $\mathcal C_1(p_m)$ in a polylogarithmic number of steps, and thus not contributing much to the diameter of $\cT$ by Wilson's algorithm.
The proof of Theorem \ref{T:highDisorder} is concluded in Section \ref{SS:Thm_high} as a consequence of Proposition \ref{P:DiaC1} and by choosing a specific value of $p_0$ in the critical window.

\smallskip

In Section \ref{SS:veryHigh}, we give an alternative proof in the case of very high disorder, that is when $\beta_n$ grows fast enough. This approach has the advantage of being applicable to any underlying graph $G$.

\medskip

{\bf Intermediate regime:} Section \ref{S:Conclusion} is devoted to the heuristics behind Conjecture \ref{C:Intermediate}. In short, we will apply Proposition \ref{P:DiaC1} with a choice of $p_0$ slightly above the critical window of Erdős-R\'enyi random graph. Thanks to the results of \cite{DKLP11}, we can then reduce to the study of the diameter of the RSTRE on random regular graphs, which could hopefully be done using the techniques developed in \cite{MSS23}.

\medskip

{\bf Note:} Shortly after we posted our paper, K\'{u}sz posted a paper \cite{K24} which introduces the same random spanning tree model with essentially the same result on the diameter as in Theorem \ref{T:main}. The key difference between our approaches is that K\'{u}sz uses the Aldous-Broder algorithm to analyse the UST on weighted graphs while we use Wilson's algorithm. K\'{u}sz studies the high disorder regime in Theorem \ref{T:main} by comparing the Aldous-Broder algorithm for UST with Prim's algorithm for MST, which reveals finer transitions within the high disorder regime.



\section{Preliminaries} \label{S:Notation}

In this section, we gather some preliminary material. In Section \ref{SS:ElectricRW} we recall some results about electrical networks and random walks; in Section \ref{SS:Wilson} we describe Wilson's algorithm for constructing the UST and recall the spatial Markov property of USTs; Section \ref{SS:Mixing} is on the bottleneck ratio of weighted random graphs and its application to bound the mixing time of the associated random walk. In Section \ref{SS:DiamterCondition} we recall the conditions
formulated in \cite{MSS23} to control the diameter of the UST on weighted random graphs. Lastly in Section \ref{SS:ERintro}, we collect several results about Erdős-R\'enyi random graphs that will be used later in the proof for the high disorder regime. 

\subsection{Electrical networks, random walks and UST} \label{SS:ElectricRW}

An electrical network is a weighted graph $(G, \weight)=((V, E), \weight)$ with non-negative weights $\weight=(\weight(e))_{e\in E}$ on the unoriented edges $e=(x,y)=(y,x)\in E$.  The edge weights $\weight(e)$ are known as conductances, and their reciprocals $r(e):=1/\weight(e)$ are called resistances. 
In our setting, we shall choose weights given by 
\begin{equation*}
	\weight(e) = \weight_\beta(e) = \exp(-\beta \omega_e),
\end{equation*}
where $\omega_e$ are i.i.d.\ random variables with common distribution $\mu$. This induces the weighted UST measure $\bP^\weight_G$ on the set of spanning trees $\mathbb{T}_G$ with
\begin{equation} \label{eq:defUST}
	\bP^\weight_G(\cT=T) = \frac{1}{Z^{\weight}_G} \prod_{e \in T} \weight(e), \qquad Z^{\weight}_G:= \sum_{T \in \mathbb{T}_G} \prod_{e \in T} \weight(e)\,.
\end{equation}
Furthermore, we can define a lazy random walk on the weighted graph $(G,\weight)$, whose law we denote by $Q$, with transition probabilities
\begin{equation*}
	q(u,v) =
	\begin{cases}
		\frac{1}{2} & \text{if } u=v \\
		\frac{1}{2} \frac{\weight(u,v)}{\sum_{x \sim u} \weight(u,x)} =: \frac{1}{2} \frac{\weight(u,v)}{\weight(u)} & \text{if } u\neq v\,,
	\end{cases}
\end{equation*}
where $x \sim u$ denotes that there is an edge $(x,u)$ in $G$ between $x$ and $u$ and for $u\sim v$ we write $\weight(u,v):=\weight((u,v))$. 
Furthermore, for a vertex $u \in V$, we have slightly abused notation by defining 
$\weight(u) := \sum_{x \sim u} \weight(u,x)$. 
The $t$-step transition probabilities of the lazy random walk will be denoted by $q_t(\cdot, \cdot)$.
Note that the lazy random walk on $(G,\weight)$ is aperiodic and has a stationary distribution $\pi$ satisfying 
\begin{equation*}
	\pi(v) = \frac{\weight(v)}{2 \sum_{e \in E} \weight(e)} \qquad \text{and} \qquad \pi(S) := \sum_{v \in S} \pi(v) = \frac{\sum_{v \in S}\weight(v)}{2 \sum_{e \in E} \weight(e)} 
\end{equation*}
for any $S\subseteq V$.
We also define $\pi_{\min} := \min_{v \in V} \pi(v)$ and $\pi_{\max} := \max_{v \in V} \pi(v)$.

There are several equivalent definitions of effective resistance, but for our purposes, we follow \cite[Chapter 2]{LP16}. 
For disjoint sets $A, B \subseteq V$, we define the effective resistance between $A$ and $B$ by 
\begin{equation*}
	\effR{\weight}{A}{B} := \frac{1}{\sum_{a \in A} \weight(a) Q_{a}(\tau_B < \tau_A^+)}, 
\end{equation*}
where $Q_a$ is the law of the lazy random walk started at $a$, $\tau_B$ is the first hitting time of $B$, and $\tau_A^+$ is the first return time to $A$. When $A=\{a\}$ and $B = \{b\}$ are singletons, we will just write
$\effR{\weight}{a}{b}$.

We will need the following classic results and refer to \cite[Sections 4.2, 2.3  and 2.4]{LP16} for more details. In particular, Theorem \ref{T:Kirchhoff} draws a connection between UST and random walk through the effective resistance of the graph.

\begin{theorem}[Kirchhoff's Formula] \label{T:Kirchhoff}
	Given any finite electrical network $(G, \weight)$ and corresponding weighted UST measure $\bP^\weight_G$, we have
	\begin{equation} \label{eq:p_in_tree_r_eff}
		\bP^\weight_G \big( (u,v) \in \cT \big) = \weight(u,v) \effR{\weight}{u}{v}\,.
	\end{equation}
\end{theorem}

\begin{lemma}[Series Law] \label{L:SeriesLaw}
	Suppose $v \in V$ is a vertex of degree $2$ with incident edges $e_1 = (u_1, v)$ and $e_2 = (v, u_2)$ that have resistances $r_1 := 1/\weight(e_1)$ and $r_2 := 1/\weight(e_2)$, respectively. Consider the  graph $G'$
	on $V \setminus \{v\}$  where $e_1$ and $e_2$ are replaced by a single edge $e$ between $u_1$ and $u_2$, with resistance $r(e) := r_1 + r_2$. Then, for any disjoint $A,B\subseteq V\setminus\{v\}$, the effective resistance between $A$ and $B$ in $G$ is equal to that in $G'$.
\end{lemma}

\begin{theorem}[Rayleigh’s Monotonicity Principle] \label{L:Rayleigh}
	Let $\weight, \weight'$ be weights on $G$ with $\weight(e) \leq \weight'(e)$ for all edges $e \in E$. If $A,B \subseteq V$ are disjoint sets, then
	\begin{equation*}
		\effR{\weight'}{A}{B} \leq \effR{\weight}{A}{B},
	\end{equation*}
	i.e.\ the effective resistance is a non-increasing function of the weights.
\end{theorem}


\subsection{Wilson's algorithm and spatial Markov property}\label{SS:Wilson}

We briefly recall Wilson's algorithm to construct the UST on a finite connected graph. This requires the notion of \textit{loop erased random walk} (LERW). Assume we are given a trajectory $\Vec{x} = (x_0, \ldots, x_m)$. 
The loop erasure $\Vec y$ of $\Vec x$ erases any cycles (in chronological order) that appear in $\Vec{x}$. More precisely, we follow the sequence $x_0, x_1, \dots$ and the first time we encounter a loop of the form $u, v, w, ..., u$, we erase $(v, w, ..., u)$ from the trajectory and then repeat the process until we reach $x_m$. 
If $X=(X_0,\dots, X_m)$ is a random walk on $(G,\weight)$, then $Y$ obtained from $X$ through this procedure is called the loop erased random walk.

Wilson's algorithm works as follows: first choose an arbitrary ordering $\{v_1, \ldots, v_n\}$ of the vertices of the given graph $G$ and set $T(0) = (\{v_1\},\emptyset)$. Then run a random walk $X$ started at $v_2$ until $X$ hits the vertex set of $T(0)$. Update the tree by setting $T(1)$ to be the union of $T(0)$ and the vertices and edges in the trajectory of the loop erasure of $X$. Repeat this procedure for each vertex $v_i$, where for the $i$-th step we run a random walk $X$ started at $v_i$ until $X$ hits the vertex set of $T(i-1)$ and update the tree by letting $T(i)$ be the union of $T(i-1)$ and the (possibly empty if $v_i \in T(i-1)$) path of the loop erasure of $X$. In the final step, the algorithm outputs a spanning tree $\cT = T(n)$. The following theorem first appeared in \cite{Wil96}.

\begin{theorem}
	Let $(G, \weight)$ be a finite connected weighted graph. Wilson's Algorithm samples a spanning tree $\cT$ according to the weighted UST measure $\bP^\weight_G(\cdot)$.
\end{theorem}


We state one more useful property of the UST measure, called the spatial Markov property, and refer e.g.\ to \cite[Sec.\ 2.2.1]{HN19} for more details. Consider a graph $G = (V,E$). The \textit{deletion} of an edge $e \in E$ produces a new graph $G - \{e\}$ obtained by removing the edge $e$ from $E$. Similarly, the \textit{contraction} of $e$ produces the graph $G / e$ obtained by first deleting $e$ and then identifying the endpoints of $e$ as a single vertex. It can be shown that the deletion and contraction operations commute, and for disjoint $A, B \subseteq E$ we may define $(G - B) / A$ as the graph obtained by removing edges in $B$ and contracting edges in $A$. 

We remark that $(G - B) / A$ may contain loops and multiple edges between vertices. However, one can define the weighted UST measure $\bP^{\weight}_{(G-B)/A}$ on $(G - B) / A$ similarly to \eqref{eq:defUST} by retaining the original edge weights (if the edge is not removed) and considering all spanning trees $\mathbb{T}_{(G - B) / A}$, where each multiple edge between a pair of vertices gives rise to a different spanning tree.

\begin{lemma}[Spatial Markov Property] \label{L:USTmarkov}
	Let $(G, \weight)$ be a finite connected weighted graph and let $A, B \subseteq E$ be two disjoint sets of edges such that $\bP^\weight_G(A \subseteq \cT, B\cap \cT=\emptyset)>0$. Then for any set of edges $F \subseteq E$,
	\begin{equation*}
		\bP^\weight_G \big( \cT = F \mid A \subseteq \cT, B\cap \cT =\emptyset \big) = \bP^{\weight}_{(G-B)/A} \big( \cT \cup A = F \big).
	\end{equation*}
\end{lemma}
\noindent
In other words, conditioned on $\cT$ containing all edges in $A$ and none of the edges in $B$, the law of $\bP^\weight_G$ restricted to the edges in $E \setminus (A \cup B)$ is equal to the weighted UST measure on $(G-B)/A$.

\subsection{Bottleneck ratio and mixing time} \label{SS:Mixing}

Given a finite weighted graph $(G,\weight)$, define the {\em bottleneck ratio} of a set $S \subseteq V$ by
\begin{equation}\label{D:bottleneck}
	\Phi_{(G, \weight)}(S)
	:=  \frac{\sum_{x\in S, y\in S^c} \pi(x) q(x,y)}{\pi(S)},
\end{equation}
where $q(\cdot, \cdot)$ and $\pi(\cdot)$ are respectively the transition probabilities and stationary distribution of the lazy random walk on $(G, \weight)$ introduced in  Section \ref{SS:ElectricRW}. In the literature, this quantity is sometimes also called the conductance of $S$, and we remark that $\Phi_{(G, \weight)}(S)$ can also be written in the form
\begin{equation}\label{eq:PhiGS}
	\Phi_{(G, \weight)}(S) 
	= \frac{1}{2} \frac{\sum_{e \in E(S, S^c)} \weight(e)}{\sum_{e \in E(S,V) } \weight(e)},
\end{equation}
where $E(A,B)$ denotes the set of edges between vertex sets $A$ and $B$. The factor $1/2$ comes from the fact that transition kernel $q(\cdot,\cdot)$ in \eqref{D:bottleneck} is that of the lazy random walk. 

Given $b > 0$, we say a graph $G$ with $n$ vertices is a $b$-{\em expander}\footnote{We note that for bounded-degree graphs sometimes the expansion ratio $|E(S, S^c)| / |S|$ is used instead, see for example \cite{MSS23}.}  if
\begin{equation}
	\frac{|E(S,S^c)|}{|E(S,V)|} \geq b \quad \forall S \text{ with } 1 \leq |S| \leq \frac{n}{2}, \label{eq:expander}
\end{equation}
or equivalently
\begin{equation}\label{eq:expander2}
	|E(S,S^c)| \geq b \min \big\{ |E(S,V)|, |E(S^c,V)| \big\} \quad \forall S \subseteq V\,.
\end{equation}
We remark that in the constant weight case, the ratio in \eqref{eq:expander} is, up to a factor of $1/2$, equal to \eqref{eq:PhiGS}. 
We further say that $(G_n)_{n \geq 1}$ is a sequence of $b$-expander graphs if each $G_n$ is a $b$-expander on $n$ vertices.

\begin{definition}[Bottleneck Ratio] \label{D:Isoperi}
	Let $(G, \weight)$ be a finite connected weighted graph, and let $\pi$ be the stationary distribution of the lazy random walk on $(G,\weight)$. The bottleneck ratio of $(G, \weight)$ is defined as 
	\begin{equation*}
		\Phi_{(G, \weight)}  := \min\limits_{0 < \pi(S) \leq 1/2} \Phi_{(G, \weight)}(S).
	\end{equation*}
\end{definition}

Furthermore, we define the (uniform) mixing time of the lazy random walk on $(G, \weight)$ as
\begin{equation*}
	\tmix(G, \weight) := \inf \left\{ t \geq 0 : \max\limits_{u,v \in V} \left| \frac{q_t(u,v)}{\pi(v)} - 1\right| \leq \frac{1}{2}\right\},
\end{equation*}
where we recall that $q_t( \cdot , \cdot )$ denotes the $t$-step transition kernel of the lazy random walk on $(G, \weight)$. There is a strong connection between the mixing time and the bottleneck ratio of a graph, which goes back to the results of \cite{SJ89}. We shall use the following theorem, which also gives bounds on the lazy random walk transition probabilities.

\begin{theorem} \label{T:HeatCheeger}
	Let $(G,\weight)$ be a finite connected weighted graph with $G=(V,E)$. Then for any $t \geq 1$
	\begin{equation*}
		q_t(u,v) \leq \frac{\pi(v)}{\pi_{\min}}\exp\Big(-\frac{\Phi_{(G,\weight)}^2}{2} t \Big)  + \pi(v),
	\end{equation*}
	and
	\begin{equation*}
		\tmix(G,\weight) \leq \frac{2}{\Phi_{(G,\weight)}^2} \log(2 / \pi_{\min})\,.
	\end{equation*}
\end{theorem}
\begin{proof}
	Equation (12.13) in \cite{LP17} shows that
	\begin{equation*}
		\Big| \frac{q_t(u,v)}{\pi(v)} - 1 \Big| \leq \frac{\e^{-\gamma_* t}}{\pi_{\min}},
	\end{equation*}
	where $\gamma_*$ is the absolute spectral gap of the transition matrix of the random walk. The theorem follows readily by the bound 
	\begin{equation*}
		\frac{\Phi_{(G, \weight)}^2}{2} \leq \gamma_* \leq 2 \Phi_{(G, \weight)},
	\end{equation*}
	for lazy chains, see e.g. \cite[Theorem 13.10]{LP17}.
\end{proof}

\subsection{Sufficient conditions for bounding the diameter} \label{SS:DiamterCondition}

It is a classical fact, see \cite{Sze83}, that the diameter of the UST on $K_n$, the complete graph on $n$ vertices, is of order $n^{1/2}$ with high probability. More recently in \cite{MNS21}, the authors showed that there exist general conditions on the lazy random walk on $G$ that imply that the diameter is of order ${n}^{1/2}$.

The result in \cite{MNS21} involves finite unweighted graphs $G=(V, E)$ with $|V|=n$. Under three conditions on $G$, the authors showed that the UST on $G$ has a diameter of order ${n}^{1/2}$ with high probability. We state here the analogue of their conditions, with parameters $D, \alpha, \theta > 0$, for a weighted graph: 
\begin{enumerate}
	\item $(G, \weight)$ is balanced:
	\begin{equation} \label{eq:balanced}
		\frac{\max_{u\in V} \pi(u)}{\min_{u\in V} \pi(u)}= \frac{\max_{u \in V} \sum_{v} \weight(u,v)}{\min_{u \in V} \sum_{v} \weight(u,v)} \leq D;
	\end{equation}
	\item $(G, \weight)$ is mixing:
	\begin{equation} \label{eq:mixing}
		\tmix(G, \weight) \leq n^{\frac{1}{2} - \alpha};
	\end{equation}
	\item $(G, \weight)$ is escaping:
	\begin{equation} \label{eq:escaping}
		\sum_{t=0}^{\tmix(G,\weight)} (t+1) \sup_{v \in V}q_t(v,v) \leq \theta.
	\end{equation}
\end{enumerate}
In \cite{MNS21}, the bound on the diameter of the UST on $G$ was formulated in terms of constants $D, \alpha, \theta$ not depending on $n$. We will make use of the following extension of \cite[Theorem 1.1]{MNS21} that includes weighted graphs and, although we will not need it in this work, also allows $D$ and $\theta$ to depend on $n$. This was formulated in \cite[Theorem 2.3]{MSS23}.

\begin{theorem}\label{T:Nach}
	For any $\alpha>0$, there exist $C, k, \gamma > 0$ such that the following holds. If $(G, \weight)$ is a weighted graph satisfying conditions \eqref{eq:balanced}-\eqref{eq:escaping} for the given $\alpha$ and for some $D,\theta \leq n^{\gamma}$, then for any $\delta > n^{-\gamma}$,
	\be\label{eq:Nach1}
	\bP^\weight_G \Big( (CD \theta \delta^{-1})^{-k} \sqrt{n} \leq {\rm diam}(\cT) \leq (CD \theta \delta^{-1})^{k} \sqrt{n} \Big) \geq 1 - \delta.
	\ee
\end{theorem}

\subsection{Erdős-R\'enyi random graphs} \label{SS:ERintro}
For the high disorder regime in Theorem \ref{T:main}, we will use a  coupling between the random environment $\omega$ and the following random graphs. 
For $p \in [0,1]$, let $G_{n,p}$ be the Erdős-R\'enyi random graph obtained by independently keeping each edge of the complete graph $K_n$ with probability $p$ and deleting it otherwise. If the edge $e$ is kept, we say it is $p$-open; otherwise, we say it is $p$-closed. For two vertices $u$ and $v$, we denote by $\{u \xleftrightarrow{p} v\}$ the event that $u$ and $v$ are connected by a path along which each edge is $p$-open. This gives rise to a set of disjoint connected $p$-clusters $\mathcal{C}_1(p), \mathcal{C}_2(p), \ldots$ each viewed as a subgraph of $K_n$. We order them by decreasing cardinality, i.e.\ $|\mathcal{C}_1(p)| \geq |\mathcal{C}_2(p)| \geq \ldots$, where $|\mathcal{C}_i(p)|$ denotes the number of vertices in the $p$-cluster. 

We will often consider a parameter of the form $p = (1 + \eps)/n$, where $\eps := \eps(n)$ is a sequence tending to zero with $\eps \geq n^{-1/3}$. For such a choice of $\eps$, the Erdős-R\'enyi random graph is either inside the so-called critical window (if $\eps = O(n^{-1/3})$) or slightly above the critical window (if $\eps \gg n^{-1/3}$). Similarly, for $p = (1- \eps)/n$ the random graph is either in the critical window or slightly below the critical window, depending on whether $\eps = O(n^{-1/3})$ or $\eps \gg n^{-1/3}$.

For a connected graph $G=(V,E)$, the quantity ${\rm Exc }(G) := |E| - |V|$ is sometimes called the \textit{excess} of the graph.  We first recall a bound on the size and excess of the largest component in the (slightly) supercritical regime, which is a direct consequence of Theorem 15 and Theorem 20 of \cite{ABR09}.
\begin{theorem} \label{T:sizeSup}
	Let $p = \frac{1 + \eps}{n}$ with $\eps =\eps(n) = o(1)$. There exists a $K > 1$ such that for $\eps n^{1/3} > K$ and $n$ large enough,
	\begin{equation*}
		\P \Bigg( \big| \mathcal{C}_1(p) \big| \not\in \Big[\frac{3}{2}\eps n, \frac{5}{2} \eps n\Big] \, \,\text{ and }\, \, {\rm Exc } \big( \mathcal{C}_1(p) \big) > 150 \eps^3 n\Bigg) \leq 2 \e^{-\eps n^{1/3}}.
	\end{equation*}
\end{theorem}
\noindent We remark that, above the critical window, the typical size of the largest component is $(2+o(1))\eps n$, i.e.\ the constants $3/2$ and $5/2$ are not optimal in the above theorem, but give bounds with high probability.  

Next, consider the following bound on the size of the largest component in the (slightly) subcritical regime.
\begin{theorem}\label{T:sizeSub}
	Let $p = \frac{1 - \eps}{n}$ with $\eps =\eps(n)= o(1)$. There exists a $K > 1$ such that for $\eps n^{1/3} > K$, $n$ large enough and any $k \geq 2$
	\begin{equation}
		\P \Big( \big| \mathcal{C}_1(p) \big| > 2 \big( k +  \log(\eps^3 n) \big) \eps^{-2} \Big) \leq 3 \e^{-k}.
	\end{equation}
\end{theorem}
\noindent The above theorem is part of \cite[Theorem 19]{ABR09}, and goes back to \cite[Lemma 1]{Luc90}. In the slightly subcritical regime considered in Theorem \ref{T:sizeSub}, the size of $\mathcal C_1(p)$ is typically of order $2  \log(\eps^3 n) \eps^{-2}$.\footnote{We note that the typical component size in \cite[Theorem 19]{ABR09} is missing the factor $2$.}

Combining Theorem \ref{T:sizeSup} and Theorem \ref{T:sizeSub} also gives us quantitative bounds on the size (optimal up to a constant) of the 2nd largest cluster in the (slightly) supercritical regime, see also \cite[Theorem 3]{Luc90}. 

\begin{theorem} \label{T:size2ndSup}
	Let $p = \frac{1 + \eps}{n}$ with $\eps = \eps(n)=o(1)$.  There exists a $K > 1$ such that for $\eps n^{1/3} > K$, $n$ large enough and any $k \geq 2$, 
	\begin{equation}
		\P \Big( |\mathcal{C}_2(p)| > 10  \big(k + \log(\eps^3 n) \big)  \eps^{-2} \Big) \leq 3 \e^{-k} + 2\e^{-\eps n^{1/3}} .
	\end{equation}
\end{theorem}

\begin{proof}

	Let $\ell \in \mathbb{N}$ with $3\eps n/2 \leq \ell \leq 5 \eps n/2$ and condition on the event $\{|\mathcal{C}_1(p)| = \ell\}$. Consider the random graph $G'$ obtained from $G_{n,p}$ by deleting all vertices in $\mathcal{C}_1(p)$ and any edge that has at least one endpoint in $\mathcal{C}_1(p)$. The graph $G'$ has $n-\ell$ vertices, with $n/2 \leq n-\ell \leq n - 3 \eps n/2$, and is distributed as $G_{n-\ell,p}$ conditioned on the event
	\begin{align*}
		\mathcal{A}_\ell : = \{ \text{the largest component of } G_{n-\ell,p} \text{ contains at most } \ell \text{ vertices} \}\,.
	\end{align*}
	Let $\mathcal{C}_1(G_{n-\ell,p})$ denote the largest component of $G_{n-\ell,p}$. Using Theorem \ref{T:sizeSup}, we have
	\begin{align}\label{jamal}
		&\P \big( |\mathcal{C}_2(p)| > 10 (k + \log(\eps^3 n)) \eps^{-2} \big) \nonumber\\
		&\qquad\qquad \leq \sum_{\ell=3\eps n/2}^{5\eps n/2} \P \Big( |\mathcal{C}_2(p)| > 10  \big(k + \log(\eps^3 n) \big) \eps^{-2} \, \Big| \, |\mathcal{C}_1(p)| = \ell \Big)\P(|\mathcal{C}_1(p)| = \ell) \nonumber\\
		&\qquad\qquad \qquad \qquad + \P\Big(|\mathcal{C}_1(p)| \not\in \big[\frac{3}{2}\eps n, \frac{5}{2} \eps n \big] \Big) \nonumber\\
		&\qquad\qquad\leq \sum_{\ell=3\eps n/2}^{5\eps n/2}  \frac{\P\Big( | \mathcal{C}_1(G_{n-\ell,p})| > 2 \big( k + \log\big( \eps^3_* (n-\ell) \big) \big) \eps^{-2}_* \Big) }{\P(\mathcal{A}_\ell)} \P(|\mathcal{C}_1(p)| = \ell) \nonumber\\
		& \qquad\qquad \qquad \qquad + 2 \e^{-\eps n^{1/3}}
	\end{align}
	with $\eps_*=\eps/2$. Call $n_*:=n-\ell$ and $p_*=\tfrac{1-\eps_*}{n_*}$.
	We  notice that
	\begin{align*}
		p = \frac{1 + \eps}{n} = \frac{1 + \eps}{n-\ell} \frac{n-\ell}{n} \leq  \frac{1 + \eps}{n-\ell} \Big( \frac{n - \frac{3}{2} \eps n}{n} \Big) \leq \frac{1 - \frac{1}{2} \eps}{n-\ell}=p_*.
	\end{align*}
	Hence we can apply Theorem \ref{T:sizeSub} with $\eps_*$, $n_*$ and $p_*\geq p$ to see that 
	\begin{align*}
		\P(\mathcal{A}_\ell ) \geq 1/2\,,
	\end{align*}
	since $\ell \geq 3\eps n/2 \geq 2(2+\log(\eps_*^3 n_*)\eps_*^{-2}$ for $K$ large enough. Applying again Theorem \ref{T:sizeSub}  with $\eps_*$, $n_*$ and $p_*$ to the numerator appearing on the r.h.s.~of \eqref{jamal} concludes the proof.
\end{proof}

Lastly, we need the following theorem about the diameter of the largest component in the critical window.
\begin{theorem}[{\cite[Theorem 1.1]{NP08}}] \label{T:diamCriticalp}
	Let $p = (1 + \lambda n^{-1/3})/n$ with $\lambda \in \R$. Then for any $\delta > 0$ there exists $A = A(\delta, \lambda)$ such that, for all large $n$,
	\begin{equation*}
		\P \Big( {\rm diam}(\mathcal{C}_1(p) ) \in \big[A^{-1} n^{1/3}, A n^{1/3} \big] \Big) \geq 1 - \delta.
	\end{equation*}
\end{theorem}



\section{Low disorder regime}  \label{S:Low}

The goal of this section is to prove Theorem \ref{T:main} in the low disorder regime, which we generalize to expander graphs in Theorem \ref{T:GeneralLow}.

\begin{theorem} \label{T:GeneralLow}
	Suppose $G_n = (V_n, E_n)$, with $|V_n|=n$, is a sequence of $b$-expanders (see \eqref{eq:expander}) with minimum degree  $d_{\min} = d_{\min}(n)$, maximum degree $d_{\max} = d_{\max}(n)$, and uniformly bounded ratio $d_{\max}/d_{\min} \leq \Delta$. Let $\cT^\omega_{n, \beta_n}$ be the RSTRE on the weighted graph $(G_n,\weight_n)$ with i.i.d.\ disorder variables $(\omega_e)_{e\in E_n}$ uniformly distributed on $[0,1]$, as in Theorem \ref{T:main}. There exists a constant $C_1 = C_1(b,\Delta)$ such that if $1\leq \beta_n \leq C_1 d_{\min}/ \log n$ and $\delta > 0$, then
	\begin{equation*}
		\widehat \P \left( C_2^{-1} \sqrt{n} \leq {\rm diam}(\cT_{n,\beta_n}^\omega) \leq C_2  \sqrt{n}  \right) \geq 1-\delta,
	\end{equation*}
	for some  $C_2 = C_2(\delta, b, \Delta) > 0$ and $n$ large enough.
\end{theorem} 

Theorem \ref{T:GeneralLow} immediately implies Theorem \ref{T:main} by noting that its assumptions are satisfied by the complete graph $K_n$ with $\Delta=1$ and, for example, $b=2/3$.

\begin{remark} \label{R:fixed_beta}
	The condition $\beta_n\geq 1$ in Theorem \ref{T:GeneralLow} is only needed to simplify the proof. The theorem holds in fact as soon as $\sup_n \beta_n<\infty$. Indeed, when $\beta_n = O(1)$, the ratio of any two edge weights $\weight_n(e)/\weight_n(e')$ is uniformly bounded from below and above, so that the stationary distribution and bottleneck ratio of $(G_n, \weight_n)$ agree with those of the unweighted graph $G_n$ up to a constant factor. Thus one may easily verify the conditions of Theorem \ref{T:Nach}. 
	Actually, as long as $\beta_n \ll \log n$, the ratio of edge weights is asymptotically $o(n^\eps)$ for any $\eps>0$, and hence the stationary distribution and bottleneck ratio agree up to factors of $n^{o(1)}$ with those of the unweighted graph $G_n$. Applying Theorem \ref{T:Nach}, the bounds on the diameter in Theorem \ref{T:GeneralLow} then hold up to factors of $n^{o(1)}$.
\end{remark}

In view of Theorem \ref{T:Nach}, in order to prove Theorem \ref{T:GeneralLow}, it suffices to show that, for $\beta_n$ growing slowly enough, the three conditions in \eqref{eq:balanced}, \eqref{eq:mixing} and \eqref{eq:escaping} are satisfied for the lazy random walk on $(G_n,\weight_n)$. The conditions will be proved via concentration arguments.  
In Section \ref{SS:expansion_concentrate}, we show concentration of the bottleneck ratio around its mean; in Section \ref{SS:proofLow} we verify the conditions of Theorem \ref{T:Nach} using the bottleneck ratio; lastly, in Section \ref{SS:graphon}, we will prove Proposition \ref{P:GHPbeta}, i.e.\ we show that for fixed $\beta$ the RSTRE converges in Gromov-Hausdorff-Prokhorov distance to Aldous' Brownian CRT.

\smallskip

From now on, we will assume that the weighted graphs $(G_n, \weight_n)$ satisfy the hypotheses of Theorem \ref{T:GeneralLow}, i.e.\ we consider a sequence of $b$-expander graphs with a ratio of maximum to minimum degree of at most $\Delta$.

\subsection{Concentration of stationary distribution and bottleneck ratio} \label{SS:expansion_concentrate}
We will make use of the following concentration inequality (see \cite[Theorem 2.8.4]{Ver18}). 
\begin{theorem}[Bernstein's inequality]\label{T:Bernstein}
	Let $X_1, \ldots, X_m$ be i.i.d.\ bounded random variables with $|X_i| \leq K$, mean $\xi$ and variance $\sigma^2$. Then for $S_m = \sum_{i=1}^m X_i$ and any $\delta > 0$, one has
	\begin{equation*}
		\P(|S_m - m \xi| \geq \delta m) \leq 2 \exp\left(-\frac{m \delta^2}{2 \sigma^2 + 2K\delta/3}\right)\,.
	\end{equation*}
\end{theorem}

We will apply this inequality to $S_m= \sum_{i=1}^m \weight(e_i)$, the sum of $m$ i.i.d.~edge weights, given by $\weight(e) = \weight_{n}(e) = \exp(-\beta_n \omega_e)$, where $\omega_e$ are i.i.d.\ random variables uniformly distributed on $[0,1]$. The probability density function of $\weight(e)$ is given by
\begin{equation*}
	f_{\beta_n}(x) = \frac{1}{\beta_n x} 1_{\{\exp(-\beta_n) \leq x \leq 1\}},
\end{equation*}
with mean 
\begin{equation*}
	\xi = \xi(\beta_n) = \E[\weight(e)] = \frac{1 - \exp(-\beta_n)}{\beta_n}\leq \frac 1 {\beta_n} 
\end{equation*}
and variance 
\begin{equation*}
	\sigma^2 
	= \sigma^2(\beta_n) = \frac{1}{2\beta_n} \Big(1 - \frac{2}{\beta_n} + \e^{-\beta_n} \Big( \frac{4}{\beta_n} - \e^{-\beta_n} - \frac{2 \e^{-\beta_n}}{\beta_n} \Big) \Big) 
	\leq \frac{1}{\beta_n}. 
	\end{equation*}
	Applying Bernstein's inequality to $S_m$ with $K=1$ and $\delta = \delta_* \xi$ for some $\delta_* \in(0,1]$ gives
	\begin{equation} \label{eq:BernU}
		\P(|S_m - m \xi | \geq \delta_* m \xi) 
		\leq 2 \exp\left(-\frac{m \delta_*^2\xi^2}{2 \sigma^2 + 2\delta_*\xi/3}\right)
		\leq 2 \exp \left( - \frac{\delta_*^2 m}{9  \beta_n}\right)\,,
	\end{equation}
	where, for the second inequality, we  use the bounds  $\xi,\,\sigma^2\leq \beta_n^{-1}$ and $\beta_n\geq 1$ to deduce that
	\begin{equation*}
		\frac{ \xi^2}{2\sigma^2 + 2\delta_*\xi/3 }  
		\geq \frac{ \xi^2}{\frac{2}{\beta_n} + \frac 1{\beta_n} } 
		= \frac{ (1-\e^{-\beta_n})^2 }{3\beta_n}
		\geq  \frac{1}{9 \beta_n}\,. 
	\end{equation*}
	
	Now we use inequality \eqref{eq:BernU} to establish that, with high probability, the stationary distribution $\pi$ of the lazy random walk on $(G_n,\weight_n)$ is within a constant multiple of the uniform distribution on the vertex set.
	\begin{lemma} \label{L:WBalanced}
		Assume the same setting as in Theorem \ref{T:GeneralLow}. If $\beta_n \leq \frac{d_{\min}(n)}{72 \log n}$, then
		\begin{equation} \label{eq:stationary}
			\P \Big(  \forall v \in V_n \ : \ \frac{1}{3 \Delta n} \leq \pi(v) \leq \frac{3\Delta}{n} \Big) \xrightarrow{n \rightarrow \infty} 1 \,.
		\end{equation}
	\end{lemma}
	
	\begin{proof}
		Let $|E_n| = m$, so that $d_{\min} n/2 \leq m \leq d_{\max} n/2$. By \eqref{eq:BernU} we have that
		\begin{equation} \label{eq:totalWeight}
			\P\Big( \big|\sum_{e \in E_n} \weight(e) - m \xi \big| \geq \frac{m}{2} \xi \Big) \leq 2 \exp \left( - \frac{ m}{36 \beta_n }\right) \longrightarrow 0\,.
		\end{equation}
		For each $v \in V_n$, denote by $d_v$ its degree, and recall the notation $\weight(v) = \sum_{u \sim v} \weight(u,v)$. By a union bound and \eqref{eq:BernU}, we obtain
		\begin{equation*}
			\P\left(\exists v \in V_n : |\weight(v) - d_v \xi | \geq  \frac{d_v \xi}{2} \right) \leq 2 n\exp\Big(- \frac{ d_{\min}}{ 36 \beta_n}\Big).
		\end{equation*}
		If we assume $\beta_n \leq  \frac{d_{\min}}{72 \log n}$, then the above expression is of order $O(1/n)$. In particular, using also \eqref{eq:totalWeight},  with high $\P$-probability it holds that, for all $ v \in V_n$,
		\begin{equation} \label{eq:stationary_bound_up_low}
			\frac{ d_v }{3 d_{\max} n} \leq \frac{d_v}{3 \cdot 2 m} \leq \pi(v) = \frac{\weight(v)}{2 \sum_{e \in E} \weight(e)} \leq \frac{3 d_v}{2m} \leq \frac{3 d_v}{d_{\min} n }\,.
		\end{equation}
		Using the assumption that $d_{\max}/d_{\min} \leq \Delta$ gives the required bounds.
	\end{proof}

	We proceed by similar arguments as above to show that the bottleneck ratio $\Phi_{(G_n,\weight_n)}$, defined in \eqref{eq:PhiGS}, is comparable to the bottleneck ratio of the graph with constant weights. 
	As $G_n$ is an expander, the weighted version also has good expansion properties.
	
	\begin{lemma} \label{L:CheegerConcentrate}
		Assume the same setting as in Theorem \ref{T:GeneralLow}. There exists a $C_\Phi = C_\Phi(b, \Delta) > 0$ such that if $\beta_n \leq C_\Phi d_{\min}/\log n$, then
		\begin{equation*}
			\P \Big( \Phi_{(G_n,\weight_n)}\geq  \frac{b}{72\Delta^2} \Big) \xrightarrow{n \rightarrow \infty} 1\,.
		\end{equation*}
	\end{lemma}
	\begin{proof}
		Let $\beta_n \leq C_\Phi d_{\min}/\log n$, where we later choose $C_\Phi$ small enough, and consider $S$ with $0 < \pi(S) \leq \frac{1}{2}$. As seen in the proof of Lemma \ref{L:WBalanced} (see equation \eqref{eq:stationary_bound_up_low}), if $C_\Phi \leq 1/72$, then with high probability $\pi_{\max}:= \max_{v \in V}\pi(v) \leq 3\Delta/n$. In particular, since $\pi(S^c) \geq 1/2$, we must have 
		$
		|S^c| \geq \pi(S^c) \pi_{\max}^{-1} \geq n/6\Delta.
		$ 
		It follows that 
		\begin{equation}\label{eq:sizeS}
			|S| \leq \Big(1-\frac1 {6\Delta}\Big) n \qquad\text{ and }\qquad \frac{|S^c|}{|S|}\geq \frac 1 {6\Delta}\,.
		\end{equation}
		Hence, it suffices to consider $S \subseteq V$ with $1 \leq |S| = r \leq (1-1/6\Delta) n$. 
		
		Recall that $E(A,B)$ denotes the set of edges between the vertex sets $A$ and $B$. Define
		\begin{align*}
			m_1 &:= |E(S,S^c)|\,, \\
			m_2 &:= |E(S,V)|\,,
		\end{align*}
		and note the trivial bounds
		\begin{equation} \label{eq:trivial_m1_m2}
			|E(S^c,V)| \geq \frac{d_{\min} |S^c|}{2} \quad \text{and} \quad d_{\max} |S|  \geq m_2 \geq \frac{d_{\min} |S|}{2}.
		\end{equation} 
		As $G_n$ is a $b$-expander, it follows by \eqref{eq:expander2} that
		\begin{align}
			m_1 \geq b \min \{ |E(S,V)|, |E(S^c,V)| \} &= b m_2 \min \Big\{1, \frac{|E(S^c,V)|}{|E(S,V)|} \Big\} \nonumber \\
			&\geq b m_2 \min \Big\{1, \frac{d_{\min} |S^c|}{2 d_{\max} |S|} \Big\} \geq \frac{b m_2}{12 \Delta^2}\,.\label{eq:m1_bound}
		\end{align}
		Furthermore, the lower bound on $m_2$ in \eqref{eq:trivial_m1_m2} together with \eqref{eq:m1_bound} imply
		\begin{equation} \label{eq:m1_lower}
			m_1 \geq b \frac{d_{\min} |S|}{24 \Delta^2}\,.
		\end{equation}
		
		By the concentration inequality in \eqref{eq:BernU}, we have
		\begin{align*}
			\P\Big( \Big| \sum_{e \in E(S,S^c)} \weight(e) - m_1 \xi \Big| \geq \frac{m_1}{2}\xi\Big) &\leq 2 \exp \big( - \frac{m_1}{36 \beta_n} \big), \\
			\P\Big( \Big| \sum_{e \in E(S,V)} \weight(e) - m_2 \xi \Big| \geq \frac{m_2}{2} \xi\Big) &\leq 2 \exp\big( - \frac{m_2}{36 \beta_n} \big),
		\end{align*}
		so that with high probability (provided that $m_1,\,m_2\gg \beta_n$) 
		\begin{equation*}
			\Phi_{(G_n,\weight_n)}(S) = \frac{1}{2}\frac{\sum_{e \in E(S,S^c)} \weight(e)}{\sum_{e \in E(S,V)} \weight(e)} \geq \frac{1}{2}\frac{\frac{1}{2}m_1 \xi}{\frac{3}{2} m_2 \xi}= \frac{m_1}{6 m_2}\,.
		\end{equation*}
		More precisely, as $m_1 \leq m_2$, we have
		\begin{equation*}
			\P\Big( \Phi_{(G_n,\weight_n)}(S) < \frac{m_1}{6 m_2}\Big) 
			\leq 4 \exp \big( - \frac{m_1}{36 \beta_n} \big)\,.
		\end{equation*}
		
		By Definition \ref{D:bottleneck} of the bottleneck ratio, it holds that
		\begin{align*}
			\Big\{\Phi_{(G_n,\weight_n)}<  \frac{b}{72\Delta^2} \Big\}
			&=\Big\{\exists S \subseteq V_n,\,0<\pi(S)\leq 1/2:\;\Phi_{(G_n,\weight_n)}(S)<  \frac{b}{72\Delta^2} \Big\}\\
			&\subseteq \Big\{\exists S \subseteq V_n,\, 1 \leq |S| \leq \Big(1 - \frac{1}{6\Delta}\Big)n:\, \Phi_{(G_n,\weight_n)}(S) < \frac{m_1}{6 m_2} \Big\}
		\end{align*}
		where the inclusion follows from \eqref{eq:sizeS} and \eqref{eq:m1_bound}.
		To conclude the proof of the lemma, we use a union bound over all subsets $S$ of size $1 \leq r \leq (1-1/6\Delta) n$ and the bound $m_1 \geq brd_{\min}/24 \Delta^2$ from \eqref{eq:m1_lower} to obtain that
		\begin{align}
			\P\Big( \exists S \subseteq V_n,\, 1 \leq |S| \leq \Big(1 - \frac{1}{6\Delta}\Big)n:\, \Phi_{(G_n,\weight_n)}(S) < \frac{m_1}{6 m_2}\Big) \hspace*{-3cm}& \nonumber \\
			&\leq 4 \sum_{r=1}^{(1 - 1/6\Delta)n} \binom{n}{r} \exp \big( -  \frac{m_1}{36 \beta_n} \big) \nonumber \\
			&\leq 4 \sum_{r=1}^{(1 - 1/6\Delta)n} \Big( \frac{\e n}{r} \Big)^r \exp \Big( -  \frac{ b r d_{\min}}{864\Delta^2 \beta_n} \Big) \nonumber \\
			&\leq 4 \sum_{r=1}^{\infty} \Big( \frac{\e n}{r} \Big)^r \exp \Big( - \frac{ b r \log n }{864 \Delta^2 C_\Phi} \Big), \label{eq:pick_Cphi}
		\end{align}
		where the third line uses the upper bound of $(en/r)^r$ on the binomial coefficient. Choosing $C_\Phi$ small enough makes the expression in \eqref{eq:pick_Cphi} go to zero. 
	\end{proof}

	\subsection{Proof of Theorem \ref{T:GeneralLow}} \label{SS:proofLow}
	Using the concentration results from the previous subsection, we may now prove Theorem \ref{T:GeneralLow}.
	
	\begin{proof}[Proof of Theorem \ref{T:GeneralLow}]
		It suffices to verify that the three conditions of Theorem \ref{T:Nach} are satisfied with high probability for a sequence of random weighted graphs as in Theorem \ref{T:GeneralLow}. Let $C_\Phi > 0$ be small enough such that the concentration bounds in Lemma \ref{L:WBalanced} and Lemma \ref{L:CheegerConcentrate} hold, and assume that $\beta_n \leq C_\Phi d_{\min}/ \log n$. Then, with high probability w.r.t.\ the law of $\omega$, the balanced condition in \eqref{eq:balanced} is satisfied by Lemma \ref{L:WBalanced} with $D=9\Delta^2$.
		
		Let $M = b/72 \Delta^2$. By Lemma \ref{L:CheegerConcentrate} we furthermore have  that, with high probability,
		\begin{equation*}
			\Phi_{(G_n, \weight_n)} \geq M\,.
		\end{equation*}
		This, together with Lemma \ref{L:WBalanced}, implies by Theorem \ref{T:HeatCheeger}  that, with high probability,
		\begin{equation*}
			\tmix(G_n, \weight_n) \leq  \frac{2}{M^2} \log( 2/\pi_{\min}) \leq C_{\rm mix} \log n
		\end{equation*}
		for some large enough $C_{\rm mix}$ depending only on $\Delta$ and $b$. The mixing condition in \eqref{eq:mixing} is thus satisfied for any $\alpha < 1/2$. Additionally, 
		\begin{equation*}
			\sup_{v \in V} q_t(v,v) 
			\leq \e^{-\frac{M^2}{2}t} \frac{\pi_{\max}}{\pi_{\min}} + \pi_{\max} 
			\leq 9\Delta^2 \e^{-\frac{M^2}{2}t}  + \frac{3 \Delta}{n}\,.
		\end{equation*}
		Thus, it holds that
		\begin{equation*}
			\sum_{t=0}^{{\tmix(G_n,\weight_n)}} (t+1) \sup_{v \in V}q_t (v,v)\leq \sum_{t=0}^{ C_{\rm mix}\log n} (t+1) \Big( 9\Delta^2 \e^{-\frac{M^2}{2}t} + \frac{3 \Delta}{n} \Big) = O(1),
		\end{equation*}
		so that the escaping condition \eqref{eq:escaping} is also satisfied. By Theorem \ref{T:Nach}, this  concludes the proof of Theorem \ref{T:GeneralLow}.
	\end{proof}

	\subsection{Lack of concentration outside the low disorder regime}\label{SS:noconcentration} 
	Once $\beta_n$ grows much faster than linearly, the concentration bounds in Lemma \ref{L:WBalanced} and Lemma \ref{L:CheegerConcentrate} no longer hold. In fact,  with high probability the balanced condition in Theorem \ref{T:Nach} is no longer satisfied. Indeed, with high probability there exists a vertex $u$ such that $\omega_{(u,v)} \geq \log (n) / 2n$ for all $v \sim u$. It follows that for $n \ll \beta_n \ll n^2$ one has on the one hand $\weight(u)\leq n\exp\{-\beta_n\log n/2n\}$, and on the other hand $2 \sum_{e\in E}\weight(e)\approx n^2 \E[\weight(e)] \approx n^2/\beta_n$, so that
	\begin{equation}
		\pi_{\min} = O \Big( \frac{\beta_n \e^{-\frac{\beta_n \log n}{2n}}}{n} \Big) = o\Big( \frac{1}{n} \Big)
	\end{equation}
	with high probability. Furthermore, consider a coupling of the environment to an Erdős-R\'enyi random graph by keeping edges that have $\omega_e \leq 1/n$. 
	By applying the concentration inequality in Theorem \ref{T:Bernstein} to the weighted edges in $E(\mathcal C_1(1/n),\mathcal C_1(1/n)^c)$ and those in $E(\mathcal C_1(1/n),V_n)$, when $n\log n \ll \beta_n \ll n^{4/3}$ one obtains
	\begin{equation}
		\Phi_{(G_n, \weight_n)}\Big (\mathcal{C}_1 \Big( \frac{1}{n} \Big) \Big ) = \bigO{ \e^{-\frac{\beta_n}{n}}}
	\end{equation}
	with high probability. As a direct consequence, one has that $\Phi_{(G_n, \weight_n)}$ goes to $0$ so fast that the mixing condition in Theorem \ref{T:Nach} no longer holds.
	
	
	\subsection{Convergence to the CRT} \label{SS:graphon}
	
	We end this section with the proof of Proposition \ref{P:GHPbeta}, which states that if $\beta$ is a fixed constant as $n$ tends to $\infty$, then the RSTRE converges to Aldous' Brownian CRT.
	
	Graphons arise as continuum limits of sequences of dense graphs and can be encoded as symmetric measurable functions  $f:[0, 1]^2 \to [0,1]$, where, heuristically, $[0, 1]^2$ represents a continuum of vertices, and $f(x, y)$ represents the probability that there is an edge between $x$ and $y$. A Lebesgue-measure preserving transformation $\phi:\,[0,1]\to[0,1]$ can be interpreted as a relabeling of the vertices, and hence $f(x, y)$ and $f(\phi(x), \phi(y))$ will be regarded as the same graphon. The space of graphons is equipped with the {\em cut-distance} $\delta_{\Box}$. 
	We refer to \cite[Section 2]{AS24} or to \cite{LS06} for a more formal introduction to graphons.
	
	\smallskip
	
	Proposition \ref{P:GHPbeta} essentially follows from \cite[Theorem 1.1 and Theorem 1.3]{AS24}, which show that the rescaled UST on a sequence of (weighted) dense expanders converges to Aldous' Brownian CRT. To apply the results of \cite{AS24} it suffices to show that the weighted graph $(K_n, \weight_n)$ converges almost surely to a (connected) positive constant graphon with respect to the cut-distance.

	\begin{proof}[Proof of Proposition \ref{P:GHPbeta}]
		Let $\weight_n(e) = \exp(-\beta \omega_e)$, $e \in E_n$, be the weights of the edges. As $\weight_n(e) \geq \exp(-\beta) > 0$ and $\beta \geq 0$ does not depend on $n$, in the terminology of \cite{AS24}, almost surely $(K_n,\weight_n)_{n \geq 1}$ forms a  sequence of dense graphs that are $\exp(-\beta)$-expanders\footnote{Notice that the notion of expanders is different from \eqref{eq:expander}, see \cite[Definition 2.4]{AS24}.} with minimal weighted degree $(n-1)\exp(-\beta)$. Theorem 1.3 in \cite{AS24} then implies that for $\P$-almost every $\omega$ there is a sequence $1 \leq \alpha_n=\alpha_n(\omega) \leq \exp(\beta)$ such that 
		\begin{equation} \label{eq:graphon_alpha}
			\big( \cT_{n,\beta}^\omega, \frac{\sqrt{\alpha_n}}{\sqrt{n}} d_{\cT_{n,\beta}^\omega}, \nu_n \big) \xrightarrow{(d)} \big( \cT, d_{\cT}, \nu \big),
		\end{equation}
		with respect to the GHP distance between metric spaces. To prove Proposition  \ref{P:GHPbeta}, it thus suffices to show that $\alpha_n \rightarrow 1$ for $\P$-almost every $\omega$.  This can be achieved, thanks to \cite[Claim 7.4]{AS24} or \cite[Theorem 1.1]{AS24},  by showing that the weighted graph $(K_n, \weight_n)$ converges $\P$-almost surely to a (deterministic) constant connected graphon.
		
		Label the vertices of $V_n$ by $1, \ldots, n$. The random weighted graph $(K_n, \weight_n)$ can be represented by the graphon $M_n = M_n(\omega)$ defined by
		\begin{equation*}
			M_n(x,y) = \weight_n( {\lceil n x \rceil }, {\lceil n y \rceil }) \qquad \forall (x,y) \in [0,1]^2,
		\end{equation*}
		where we use the convention $\weight_n(i,i) = \weight_n(0,i) =\weight_n(i,0) =  0$ for all $i=0,\dots,n$. Furthermore, let $M \equiv \E[\weight_n(e)] = \E[ \exp(- \beta \omega_e)] = (1- \e^{-\beta})/\beta = \xi_\beta$ be a constant (connected) graphon. We wish to show that
		\begin{equation} \label{eq:graphonConv}
			\delta_{\Box}(M_n, M) \xrightarrow{\text{a.s.\ }} 0\,,
		\end{equation}
		which by \cite[Claim 7.4]{AS24} implies that $\alpha_n \rightarrow 1$ almost surely.
		
		To prove \eqref{eq:graphonConv}, consider the intervals $I_j := [(j-1)/n, j/n)$, $j \in \{1, \ldots, n\}$, and for any $U \subseteq [0,1]$ let $I(U) = \{ j \in \{1, \ldots, n\} : U \cap I_j \neq \emptyset\}$, i.e.\ $I(U)$ keeps track of the intervals $I_j$ the set $U$ is contained in. Denote by $\mathcal{B}([0,1])$ the Borel subsets of $[0,1]$, then the cut-distance satisfies
		
		\begin{align}
			\delta_{\Box}(M_n, M) &\leq \sup\limits_{U,V  \in \mathcal{B}([0,1]) } \bigg| \int_{U \times V} \big( M_n(x,y) - M(x,y) \big) dx dy \bigg| \nonumber \\
			&= \sup\limits_{I(U), I(V) \subseteq \{1, \ldots, n\}} \bigg|  \sum_{i \in I(U)} \sum_{j \in I(V)}\int_{I_i \times I_j} \big( M_n(x,y) - M(x,y) \big) dx dy \bigg| \nonumber \\
			&\leq \sup\limits_{I(U), I(V) \subseteq \{1, \ldots, n\}} \frac{1}{n^2} \bigg( \bigg|  \sum_{i \in I(U)} \sum_{\substack{j \in I(V) \\ j \neq i}}\Big( \exp \big(- \beta \omega_{(i,j)} \big) - \xi_\beta \Big) \bigg| + n \xi_\beta \bigg),  \label{eq:Cutgraphon}
		\end{align}
        where in the last step we used the triangle inequality together with the fact that for $x,y$ in the same interval $I_i$, $\forall i \in I(U)$, we have $|M_n(x,y) - M(x,y)| = |M(x,y)| = \xi_\beta$ since $M_n(x,y) = 0$ (there are no self-loops).
		Using Hoeffding's inequality \cite[Theorem 2.2.6]{Ver18} gives that for any $I(U)$, $I(V) \neq \emptyset$ (excluding for the moment the case $I(U) = I(V) = \{i\}$) 
		and for any $\eps > 0$
		\begin{align*}
			&\P\bigg( \bigg| \sum_{i \in I(U)} \sum_{j \in I(V)} \Big( \weight_n(i,j) - \xi_\beta \Big) \bigg| \geq \eps n^2 \bigg) \\
			&\hspace*{1cm} \leq \P\bigg( \bigg| \sum_{i \in I(U)} \sum_{\substack{j \in I(V) \\ j \neq i}} \Big( \exp \big( -\beta \omega_{(i,j)} \big) - \xi_\beta \Big) \bigg| \geq \eps n\Big(n - \frac{\xi_\beta}{\eps}\Big) \bigg)\\
			&\hspace*{1cm} \leq 2 \exp\Big(-\frac{2 \eps^2 n^2 \big(n - \frac{\xi_\beta}{\eps}\big)^2}{|I(U)| |I(V)| - |I(U) \cap I(V)|}\Big) \leq 2 \exp\Big(-2 \eps^2 \big(n - \frac{\xi_\beta}{\eps}\big)^2\Big)\,.
		\end{align*}
		When $I(U) = I(V) = \{i\}$ for some $i=1,\dots,n$, the probability is just equal to $0$ for $n$ large enough. 
		Using this bound together with a union bound over all possible sets $I(U)$ and $I(V)$, we obtain from \eqref{eq:Cutgraphon} that for $n$ large enough
		\begin{equation*}
			\P \big( \delta_{\Box}(M_n, M)  \geq \eps \big) \leq 2 \cdot 2^{2n} \exp\Big(- 2 \eps^2 \big(n - \frac{\xi_\beta}{\eps}\big)^2\ \Big).
		\end{equation*}
		As this probability is summable in $n$ for any $\eps > 0$, the Borel-Cantelli lemma gives the required convergence in \eqref{eq:graphonConv}.
	\end{proof}

	\begin{remark}
		The above proof depends on the convergence of $(K_n, \weight_n)$ to the constant graphon $M \equiv \xi_\beta$. When $\beta=\beta_n\to\infty$ one has $\xi_\beta \rightarrow 0$, so that the only possible limiting graphon would be the graphon identically equal to $0$. This case is out of the scope of \cite{AS24} for dense graphs and one can no longer use their results.
		However, in the low disorder regime with $\beta_n \leq C n/\log n$ the graphs are still weighted expanders
		in the sense of Lemma \ref{L:CheegerConcentrate}. One could then still expect the limiting object to be Aldous' Brownian CRT.
		When $\beta_n \gg n$, instead,  $(K_n,\weight_n)_{n\geq 1}$ is no longer a  sequence of weighted expanders, see Section \ref{SS:noconcentration}.
	\end{remark}
	
	
	
	\section{High disorder regime} \label{S:High}

	In this section, we prove the second part of Theorem \ref{T:main}. In Theorem \ref{T:highDisorder} below we restate the results of the high disorder regime of Theorem \ref{T:main}, adding also some information on the expected diameter of the RSTRE. The proof proceeds in several steps. In Section \ref{SS:graph_reduce}, we first introduce a coupling of the random environment $\omega$ to Erdős-R\'enyi random graphs $G_{n,p}$. Then, we present Proposition \ref{P:DiaC1}, the key step in the proof of Theorem \ref{T:main} for the high disorder regime, which states that, up to small correction terms, the diameter of the whole random spanning tree is comparable to the diameter of the minimal subtree containing all vertices in the largest component of $G_{n, p}$ for a slightly supercritical $p$ dependent on $\beta_n$. 
	In Sections \ref{SS:gap}, \ref{SS:Diam_sup} and \ref{SS:Connect_outside}, we prove several preliminary lemmas, which we use to prove Proposition \ref{P:DiaC1} in Section \ref{SS:ProofHigh}. Then in Section \ref{SS:Thm_high}, we show how Proposition \ref{P:DiaC1} implies Theorem \ref{T:highDisorder}. Lastly, in Section \ref{SS:veryHigh}, we give a simple alternative proof for the high disorder regime when $\beta_n$ is very large, which also holds for any sequence of finite graphs.

	\begin{theorem} \label{T:highDisorder}
		Let $\cT_{n,\beta_n}^\omega$ be the RSTRE on the complete graph with law $\bP^\omega_{n, \beta_n}$ as in Theorem \ref{T:main}. If $\beta_n \geq n^{4/3} \log n$, then
		\begin{equation*}
			\widehat{\E}[ {\rm diam}(\cT_{n,\beta_n}^\omega)] = \Theta(n^{1/3}).
		\end{equation*}
		Furthermore, for any $\delta > 0$ there exists a constant $C = C(\delta) > 0$ such that, for all $n \in \mathbb{N}$, 
		\begin{equation*}
			\widehat{\P} \big( C^{-1} n^{\frac{1}{3}} \leq {\rm diam}(\cT_{n,\beta_n}^\omega) \leq C n^\frac{1}{3} \big) \geq 1 - \delta.
		\end{equation*}
	\end{theorem}
	
	A key step in the proof of Theorem \ref{T:highDisorder} will be Proposition \ref{P:DiaC1}, where we prove that the order of the diameter essentially stems from the spanning tree on the largest component in the Erdős-R\'enyi random graph $G_{n,p}$ with parameter $p \geq  1/n + C\log n/\beta_n$.
	
	\subsection{Reduction via random graphs} \label{SS:graph_reduce}
	For many of the proofs, it will be convenient to couple the random environment $\omega$ with the Erdős-R\'enyi random graphs $(G_{n,p})_{p \in [0,1]}$ as follows. Let $p \in [0,1]$, then, for each edge $e$ in $K_n = (V_n, E_n)$, keep $e$ if $\omega_e \leq p$, and call the edge $p$-open, otherwise delete $e$ and call the edge $p$-closed.
	Note that conditioned on $G_{n, p}$, $(\omega_e)_{e\in E_n}$ is still a family of independent random variables with distribution given by the distribution of $\omega_e$ conditioned on whether $\omega_e \leq p$ or $\omega_e > p$. Notice that, since $\weight_n(e)=\exp(-\beta_n\omega_e)$, in this coupling small values of $p$ correspond to high conductances.
	
	\smallskip
	
	Recall from Section \ref{SS:ERintro} that, for two vertices $u,v \in V_n$, we denote by $\{u \xleftrightarrow{p} v\}$ the event that $u$ and $v$ belong to the same connected component in the Erdős-R\'enyi random graph $G_{n,p}$.
	For a given $p$, we are often interested in the paths of the RSTRE connecting two vertices in the same cluster of $G_{n,p}$, in particular when the cluster is $\mathcal{C}_1(p)$ -- the largest component of $G_{n,p}$. 
	Therefore, for any vertex set $A \subseteq V_n$, and in particular when $A=\mathcal C_1(p)$, we define  the minimal subtree of $\cT=\cT_{n,\beta_n}^\omega$ containing all vertices in $A$ as
	\begin{equation}  \label{eq:defCbar}
		\cT_A \ := \bigcup_{u,v \in A} \gamma_{\cT}(u, v),
	\end{equation}
	where $\gamma_{\cT}(u, v)$ is the unique path in $\cT$ (regarded as a graph with its own vertex and edge set) connecting $u$ and $v$.

\smallskip

The following
proposition reduces the study of the diameter of the random spanning tree $\cT_{n,\beta_n}^\omega$ on $K_n$ to the diameter of the subtree $\ClCp{1}{p_0}$ for a critical or (slightly) supercritical $p_0$. From now on we will just write $\cT=\cT_{n,\beta_n}^\omega$ to ease the notation.

\begin{proposition} \label{P:DiaC1}
	Let $\cT$ be the RSTRE on the complete graph $(K_n,\weight_n)$ with law $\bP^\omega_{n, \beta_n}$ and $\beta_n\geq n(\log n)^2$. There exists a constant $g_0 \geq 1$ such that for any $n^{-1/3} \leq \eps = \eps(n) \leq (\log n)^{-1}$, with $\beta_n \eps \geq n \log n$, the following holds. Let $p_0 = (1 + g_0 \eps)/n$, then 
	\begin{equation}
		\widehat{\E}[ {\rm diam}(\cT)] = \widehat{\E}\Big[ {\rm diam}\Big( \ClCp{1}{p_0} \Big) \Big] + \bigO{\frac{n^{1/3}}{\sqrt{\eps n^{1/3}}}}\,. \label{eq:E_DiaC1}
	\end{equation}
	Furthermore,
	\begin{equation}
		{\rm diam}\Big(\ClCp{1}{p_0} \Big) \leq  {\rm diam} (\cT ) \leq {\rm diam}\Big( \ClCp{1}{p_0} \Big) + \bigO{\frac{n^{1/3}}{\sqrt{\eps n^{1/3}}} + (\log n)^6} \label{eq:P_DiaC1}
	\end{equation}
	with $\widehat{\P}$ probability at least
	\begin{equation*}
		1 - \bigO{\exp \Big(- \frac{1}{4} \sqrt{g_0 \eps n^{1/3}} \Big)} - \bigO{\frac{1}{n}}. \label{eq:DiaC1_probbound}
	\end{equation*}
\end{proposition}

In Section \ref{SS:ProofHigh}, we will prove Theorem \ref{T:highDisorder} by applying Proposition \ref{P:DiaC1} with $\eps=n^{-1/3}$ and using the fact that when $p_0$ is in the critical window, the largest component $\cC_1(p_0)$ of the Erdős-R\'enyi  random graph is almost a tree and its diameter is known to be of order $n^{1/3}$. Notice that when $\beta_n=n^{1+\gamma}$ for some $0< \gamma< 1/3$ we can choose any $n^{-\gamma}\log n\leq \eps\leq  1/\log n$ in Proposition \ref{P:DiaC1}. This can be used to reduce the problem in the intermediate regime, see Section \ref{S:Conclusion}.

The proof of Proposition \ref{P:DiaC1} will follow a similar strategy to the proof of Theorem 1 in \cite{ABR09}, which shows that the diameter of the minimum spanning tree on $(K_n, \weight_n)$ is of order $n^{1/3}$, and we refer to Remark \ref{R:MSTdiff} for a brief discussion on the differences between the two proofs. 
We will split the proof into the following three intermediate steps, which allow us to pass from the diameter of $\cT_{\mathcal C_1(p_0)}$ to that of $\cT$:
\begin{enumerate} \label{page:proof_idea}
	\item Lemma \ref{L:gap} (Section \ref{SS:gap}): For $p < q$, with $q - p$ ``large'', and any $u,v \in V_n$ belonging to the same cluster in $G_{n,p }$, the path $\gamma_\cT(u,v)$ in the RSTRE between $u$ and $v$ typically uses only $q$-open edges. In particular, when $p=p_0$ and $q=p_1$ with $p_1-p_0\geq 6\log n/\beta_n$, one has that $\ClCp{1}{p_0}$ is contained in a connected component of  $G_{n,p_1}$ with high $\widehat\P$ probability. 
	\label{enu:gap}
	
	\item Lemma \ref{L:sizepm} (Section \ref{SS:Diam_sup}): For $p_m$ chosen such that $|\mathcal{C}_1(p_m)|$ is of order $n/\log n$, the diameter of $\ClCp{1}{p_m}$ is comparable to the diameter of $\ClCp{1}{p_0}$ with $p_0$ chosen as in Proposition \ref{P:DiaC1}.
	This is achieved via a telescopic procedure: we first choose a sequence  $p_0< p_1<\dots< p_m$ with $p_{i+1}-p_{i}\geq 6\log n/\beta_n$ (allowing us to apply Lemma \ref{L:gap}) and then control the difference between the diameters of $\ClCp{1}{p_i}$ and $\ClCp{1}{p_{i+1}}$. 
	This is where we need the assumption $\beta_n\eps\geq n\log n$ appearing in Proposition \ref{P:DiaC1}.
	
	\item Lemma \ref{L:fasthitLERW} (Section \ref{SS:Connect_outside}): The LERW started at vertices in $V_n \setminus \mathcal{C}_1(p_m)$ ``quickly'' hits $\mathcal{C}_1(p_m)$ with high probability, which by Wilson's algorithm ensures that the diameter of $\cT$ remains of the same order as $\ClCp{1}{p_m}$.
\end{enumerate}
See Figure \ref{fig:proofidea} for an illustration of the main ideas.

\begin{figure}[tb]
	\centering
	\includegraphics[width=8.5cm]{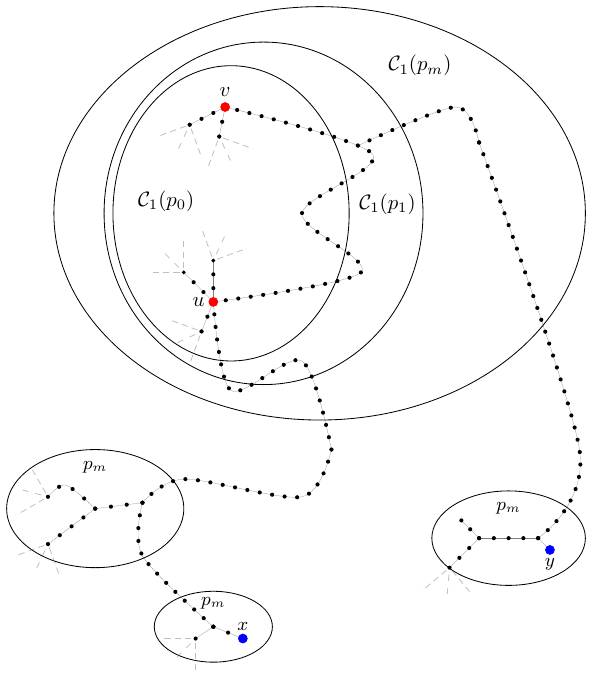}
	\caption{Typical vertices $x$ and $y$ are likely to be connected in $\cT$ via a short path from $x$ to $\ClCp{1}{p_m}$, a path inside $\ClCp{1}{p_m}$ and another short path from $\ClCp{1}{p_m}$ to $y$. }
	\label{fig:proofidea}
\end{figure}


\subsection{Connections inside \texorpdfstring{$p$}{p}-clusters} \label{SS:gap}

We start by proving step \eqref{enu:gap}, namely that for any choice of $p < q$ with $\beta_n(q - p) \geq 6 \log n$, the event $\{\ClCp{1}{p}$ is contained in a connected component of $ G_{n,q} \}$ occurs with high probability. In other words, any two points in the same cluster of $G_{n,p}$ are connected in the RSTRE via a path using only edges with $\omega_e \leq q$.

\begin{lemma} \label{L:gap}
	Given $p < q \in [0,1]$, define for $u,v \in V_n$ the event
	\begin{equation*}
		F(u,v) := \big\{ u \xleftrightarrow{p} v, \,\gamma_{\cT}(u,v) \text{ contains an edge $e$ with }  \omega_e > q \big\}\,,
	\end{equation*}
	where $\gamma_{\cT}(u, v)$ is the unique path in $\cT$ connecting $u$ and $v$.
	Then
	\begin{equation}
		\widehat{\P} \bigg( \bigcup_{u,v \in V_n, u \neq v} F(u,v) \bigg) 
		\leq n^5 \e^{-\beta_n (q - p)} \,.
		\label{eq:gapUnion}
	\end{equation}
\end{lemma}

\begin{proof}
	The event $F(u,v)$ is the union over all edges $e\in E_n$ of the events
	\begin{equation}
		F(u,v; e) 
		:=  \big\{ u \xleftrightarrow{p} v, \, e \in \gamma_{\cT}(u,v), \, \omega_e > q \big\}
	\end{equation}
	We will first analyze the special case of $e = (u,v)$, and use the formula
	\begin{equation*}
		\bP^{\omega}_{n, \beta_n}(e = (u,v) \in \cT) = \weight(u,v) \effR{\weight}{u}{v}
	\end{equation*}
	in Theorem \ref{T:Kirchhoff} to bound the probability that $e$ is in the tree.
	Condition on the realization of  $G_{n,p}$ (determined by $\omega$ via the coupling) and assume that the events  $\{ u \xleftrightarrow{p} v \}$ and $\{ \omega_e > q\}$ hold. Under $\{ u \xleftrightarrow{p} v \}$ there exists an acylic path $\gamma(u,v)$ connecting $u$ and $v$ using only edges $f$ with $\omega_f \leq p$. The series law (Lemma \ref{L:SeriesLaw}) and Rayleigh's monotonicity principle (Theorem \ref{L:Rayleigh}) give that the effective resistance can be upper bounded by the length of $\gamma(u,v)$ times the maximum resistance along $\gamma(u,v)$. For such $\omega$ we have
	\begin{align}
		\bP^{\omega}_{n, \beta_n}(e \in \gamma_\cT(u,v)) =
		\bP^{\omega}_{n, \beta_n}(e \in \cT) &\leq \e^{- \beta_n q} \effR{\weight}{u}{v} \nonumber  \\
		&\leq \e^{ -\beta_n q} \frac{n}{\e^{-\beta_n p}}  
		= n \e^{-\beta_n(q - p)}\, ,\label{eq:effR_dif}
	\end{align}
	from which it follows that $\widehat{\P}(F(u,v;e)) \leq n \exp(- \beta_n(q - p))$. 
	
	We now consider $F(u,v; e)$ with $e\neq (u, v)$.  
	For $e$ to be in $\gamma_\cT(u,v)$, there must exist (possibly empty) disjoint paths $\Gamma_{u}$ and $\Gamma_{v}$ in $\cT$ connecting $u$, respectively $v$, to either $x$ or $y$. Summing over all possible realizations of $\Gamma_u$ and $\Gamma_v$, we obtain
	\begin{equation*}
		\bP^{\omega}_{n, \beta_n}(e \in \gamma_\cT(u,v)) 
		= \sum_{\Gamma_u, \Gamma_v } \bP^{\omega}_{n, \beta_n} \big( e \in \gamma_\cT(u,v) \mid \Gamma_u, \Gamma_v \subseteq \cT \big) \bP^{\omega}_{n, \beta_n}(\Gamma_u, \Gamma_v \subseteq \cT).
	\end{equation*}
	Consider now the weighted graph $(G', \weight)$ obtained by contracting the paths $\Gamma_u$ and $\Gamma_v$ into two single vertices as in Section \ref{SS:Wilson}. By the spatial Markov property (Lemma \ref{L:USTmarkov}), conditioned on $\{ \Gamma_u, \Gamma_v \subseteq \cT \}$, the probability that $e \in \gamma_\cT(u,v)$ is equal to the probability that $e$ is in the (weighted) UST on $(G', \weight')$. 
	See also Figure \ref{fig:gap}.
	
	In the graph $(G', \weight)$ the vertices $u$ and $v$ are still connected by a $p$-open path of length at most $n$, and furthermore, $e$ is now an edge between $u$ and $v$. The same arguments as in the bound  for the special case of $e=(u,v)$ now apply. Thus, we obtain the same upper bound as in \eqref{eq:effR_dif}.
	
	Using several union bounds and \eqref{eq:effR_dif} then gives 
	\begin{equation*}
		\widehat{\P} \bigg( \bigcup_{\substack{u,v \in V_n \\ u \neq v}} F(u,v) \bigg) \leq \sum_{\substack{u,v \in V_n \\ u \neq v}} \sum_{e \in E_n} \widehat{\P} \big( F(u,v; e) \big)  \leq n^5 \e^{-\beta_n(q - p)} \, ,
	\end{equation*}
	completing the proof.
\end{proof}%

\begin{figure}[t]
	\centering
	\includegraphics[width=13cm]{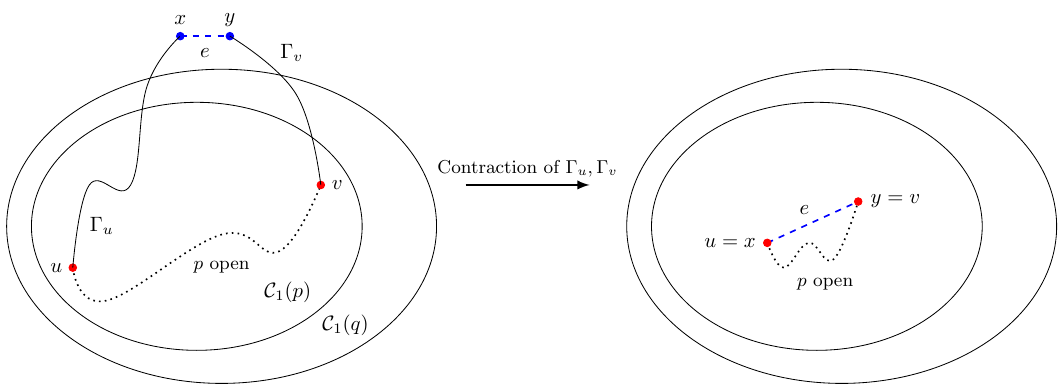}
	\caption{The path $\gamma_\cT(u,v)$ can be decomposed into a path $\Gamma_u$, an edge $e$ and another path $\Gamma_v$. After the contraction of $\Gamma_u$ and $\Gamma_v$, the edge $e$ becomes an edge between $u$ and $v$ in $(G', \weight)$.} 
	\label{fig:gap}
\end{figure}

\subsection{Diameter of the supercritical component} \label{SS:Diam_sup}

The next step is to show that for large enough $p_m$, such that $|\mathcal C_1(p_m)|$ is of order $n/\log n$, we show that the diameter of $\ClCp{1}{p_m}$ (recall \eqref{eq:defCbar}) is very close to that of $\ClCp{1}{p_0}$. The proof strategy takes inspiration from \cite{ABR09}. Recall that $\eps = \eps(n)$ is a function going to zero but such that $\eps \geq n^{-1/3}$. For some large $g_0 \geq 1$, specified later in the proofs, we consider the sequences
\begin{equation} \label{eq:g_seq}
	\begin{aligned}
		g_i &= (5/4)^{i/2} g_0, \\
		p_i &= \frac{1 + g_i \eps}{n}, 
	\end{aligned}
\end{equation}

and we let $m$ be the integer satisfying  
\begin{equation}\label{eq:definitionm}
	m = m(n) = \min \big\{ i \geq 0 \, : \, g_i \eps \geq \frac{1}{\log n} \big\}.
\end{equation}
Note that $m = O(\log n)$, and by Theorem \ref{T:sizeSup} the size of $\mathcal{C}_1(p_m)$ is of order $n/ \log n$. We will consider the sequence of random graphs $G_{n,p_i}$, $0 \leq i \leq m$, coupled to the same random environment $\omega$ as in Subsection \ref{SS:graph_reduce}.

\begin{lemma} \label{L:sizepm}
	Let $\eps \geq n^{-1/3}$ with $\beta_n\eps \geq n \log n$. Then, for some $g_0$ large enough and $p_0$ and $p_m$ defined in \eqref{eq:g_seq} and \eqref{eq:definitionm}, we have
	\begin{equation*}
		\widehat{\E} \Big[ {\rm diam} \big( \ClCp{1}{p_m} \big) \Big] \leq \widehat{\E} \Big[ {\rm diam} \big( \ClCp{1}{p_0} \big) \Big] +  \bigO{\frac{n^{1/3}}{\sqrt{\eps n^{1/3}}}}\,.
	\end{equation*}
\end{lemma}

\begin{remark} \label{R:Critical-pm}
	As will be noted just after the proof, if we choose $\eps = n^{-1/3}$ in Lemma \ref{L:sizepm} we can obtain the bound 
	\begin{equation*} 
		\widehat{\E} \Big[ {\rm diam} \big( \ClCp{1}{p_m} \big) \Big] = O(n^{1/3})\,.
	\end{equation*}
	We shall make use of this later in the proof of Theorem \ref{T:highDisorder}.
\end{remark}

To prove Lemma \ref{L:sizepm}, we need to introduce some extra notation. For two graphs $H$ and $G$, we write $H \subseteq G$ if $H$ is a subgraph of $G$. We denote by $\ell(G)$ the longest acyclic path in a connected graph $G$, and for $G_1, G_2$ subgraphs of $G$, we write $\ell(G_1 \setminus G_2)$ for the longest acyclic path in $G_1$ not using vertices in $G_2$.
Let $G_1,G_2,G_1',G_2'$ be connected subgraphs of $G$ with $G_1'$ a subgraph of $G_1$ and the vertices of $G_2'$ forming a subset of those of $G_2$. 
We shall make use of the following obvious inequalities
\begin{equation}\label{eq:lpsuper}
	\begin{aligned}
		\ell (G_1' \setminus G_2) &\leq \ell (G_1 \setminus G_2),  \\
		\ell (G_1 \setminus G_2) &\leq \ell (G_1 \setminus G_2').
	\end{aligned}
\end{equation}
Note that for vertices $u,v \in G_1$ we can construct a path between them by considering the union of a path from $u$ to some $x_1 \in G_2$, a path from $v$ to some $x_2 \in G_2$ and a path from $x_1$ to $x_2$ in $G_2$. This leads to the following upper bound on the diameter of $G_1$.
\begin{lemma}[Lemma 2 of \cite{ABR09}] \label{L:diamLP}
	If $G_2 \subseteq G_1$ are both connected, then
	\begin{equation}
		{\rm diam}(G_1) \leq {\rm diam}(G_2) + 2 \ell(G_1 \setminus G_2) + 2.
	\end{equation}
\end{lemma}
We will repeatedly apply Lemma 4.5 to bound ${\rm diam}(\ClCp{1}{p_i})$ in terms of ${\rm diam}(\ClCp{1}{p_{i+1}})$ for $0  \leq  i \leq m-1\,$ when $\,\mathcal C_1(p_i)\subset\mathcal C_1(p_{i+1})$. This requires us to control the probabilities of the following events. We say the index $i$, $0 \leq i \leq m$, is well-behaved if all the events
\begin{align}
	A(i) &:= \Big\{ |\mathcal{C}_1(p_i)| \geq \frac{3}{2} g_i \eps n\,
	\text{ and }\, \ell(\mathcal{C}_1(p_i)) \leq (g_i \eps n^{1/3})^4 n^{1/3} \Big\}, \label{eq:good1} \\
	B(i) &:= \Big\{  \ell \big( G_{n, p_{i+2}} \setminus \mathcal{C}_1(p_i) \big)\leq n^{1/3}/ \sqrt{g_i  \eps n^{1/3} } \Big\}, \label{eq:good2} \\ 
	C(i) &:= \Big\{  \mathcal{C}_1(p_i) \subseteq \mathcal{C}_1(p_{i+1}) \Big\} ,  \label{eq:good3}
\end{align}
hold. Under the event $A(i)$, we have a bound on the length of the longest path in the largest cluster of $G_{n,p_i}$. The event $B(i)$ says that paths in $G_{n, p_{i+2}}$ avoiding $\mathcal{C}_1(p_i)$ cannot be too long.

The following lemma bounds the probability that an index $i$ is not well-behaved.  
\begin{lemma} \label{L:good}
	There exists a constant $K > 1$ such that if $g_0 > K$ and $0 \leq i \leq m$, then
	\begin{align}
		\P(A(i)^c) &\leq \exp\Big( -\frac{1}{4} \sqrt{g_i  \eps n^{1/3}} \Big)\,, \label{eq:bad_complement_A}\\
		\P(B(i)^c) &\leq 5 \exp\Big( -\frac{1}{4} \sqrt{g_i  \eps n^{1/3}} \Big)\,, \label{eq:bad_complement_B} \\
		\P(C(i)^c) &\leq 2 \exp\Big( -\frac{1}{4} \sqrt{g_i  \eps n^{1/3}} \Big)\,. \label{eq:bad_complement_C}
	\end{align}
\end{lemma}

\begin{proof}
	The inequalities \eqref{eq:bad_complement_A} and \eqref{eq:bad_complement_B} follow directly from \cite[Lemma 3]{ABR09} 
	since 
	$p_{i+2}-1/n\leq 5/4 (p_{i}-1/n)$ 
	and $1/n\leq p_i\leq 1/n + 2/n\log n$. Furthermore, notice that
	\begin{equation*}
		\Big\{ | \mathcal{C}_1(p_i) | \geq   \frac{3}{2} g_i \eps n \Big\} \cap \Big\{ | \mathcal{C}_2(p_{i+1})| < \frac{3}{2} g_i \eps n \Big\} \subseteq C(i)\,.
	\end{equation*}
	It follows, by the bound in \eqref{eq:bad_complement_A} and Theorem \ref{T:size2ndSup}, say with $k=g_i \eps n/10$, that
	\begin{equation*}
		\P(C(i)^c) 
		\leq \P\big(A(i)^c\big) + \P\Big(| \mathcal{C}_2(p_{i+1})| \geq  \frac{3}{2} g_i \eps n \Big)
		\leq \e^{ -\frac{1}{4} \sqrt{g_i  \eps n^{1/3}}} +  3 \e^{ -\frac{1}{10} g_i \eps^3 n } + 2 \e^{- g_i \eps n^{1/3} },
	\end{equation*}
	which satisfies the bound in \eqref{eq:bad_complement_C} for $g_0$ large enough.
\end{proof}

\smallskip

\begin{proof}[Proof of Lemma \ref{L:sizepm}.]
	Recall the choice of $m$ from \eqref{eq:definitionm}.
	For $0 \leq i \leq m$, consider the events
	\begin{align}
		F(i) := \big\{ \ClCp{1}{p_i} \subseteq G_{n, p_{i+1}} \big\} = \big\{ \forall u,v \in \mathcal{C}_1(p_i), \forall e \in \gamma_{\cT}(u,v), \omega_e \leq p_{i+1} \big\}, \nonumber 
	\end{align}
	which states that every two vertices $u,v \in \mathcal{C}_1(p_i)$ are connected in $\cT$ via a $p_{i+1}$-open path. 
	We claim that the event 
	\begin{align}
		F :=  \bigcap_{0 \leq i \leq m} F(i) = \big\{ \forall \, 0 \leq i \leq m \, : \, \ClCp{1}{p_{i}} \subseteq G_{n, p_{i+1}} \big\}, \label{eq:F_subset}
	\end{align}
	holds with high $\widehat \P$-probability.
	Indeed, note that the complement of $F(i)$ is a subset of the event appearing in \eqref{eq:gapUnion} (with $p=p_i$ and $q=p_{i+1}$) so that by a union bound, Lemma \ref{L:gap} and the fact that $m \leq n$, we obtain 
	\begin{equation}
		\widehat{\P}(F^c) 
		\leq m  \max_{0 \leq i \leq m} \widehat{\P}(F(i)^c) 	
		\leq n^6 \exp \Big(- {\beta_n }(p_1-p_0) \Big)
		\leq n^6 \exp \Big(- \frac{g_0 \log n}{10} \Big)
		=O\Big(\frac 1n\Big) \label{eq:F_subset_bound}
	\end{equation}
	where the third inequality used the assumption that $\beta_n \eps \geq n \log n$ and the last equality holds for $g_0$ large enough.
	
	\smallskip
	
	For $0 \leq i \leq m-1$, denote now by $D_i$  the diameter of $\ClCp{1}{p_i}$. Assume that $i$ is well-behaved and that $\ClCp{1}{p_{i+1}} \subseteq G_{n, p_{i+2}}$.
	By Lemma \ref{L:diamLP} and the inequalities in \eqref{eq:lpsuper} we have
	\begin{align}
		D_{i+1} - D_i 
		&\leq 2 \ell \left( \ClCp{1}{p_{i+1}} \setminus \ClCp{1}{p_{i}} \right) + 2 \nonumber \\
		&\leq 2 \ell \left( G_{n, p_{i+2}} \setminus  \mathcal{C}_1(p_i) \right) + 2 \nonumber \\
		&\leq \frac{2n^{1/3}}{\sqrt{g_i \eps n^{1/3}}} + 2\,, \label{eq:diff_in_D}
	\end{align}
	where in the first line we used the fact that the event $C(i)$ guarantees that $\ClCp{1}{p_i}\subseteq \ClCp{1}{p_{i+1}}$, while the last inequality holds by the definition of the event $B(i)$ in \eqref{eq:good2}.
	We may repeat this argument as long as the indices are well-behaved. 
	
	\smallskip
	
	Let $i^*$, $0 \leq i^* \leq m$, be the index such that
	\begin{equation}
		i^* := \min \{ 0 \leq i \leq m \, : \, \text{all indices } i\leq k < m \text{ are well-behaved} \}\,, \label{eq:i*}
	\end{equation}
	or let $i^* = m$ if $m-1$ is not well-behaved. Under the assumption that $\ClCp{1}{p_{i}} \subseteq G_{n, p_{i+1}}$ for all $0 \leq i \leq m$, after repeating the argument leading to \eqref{eq:diff_in_D} we obtain 
	\begin{align}
		D_m - D_{i^*} &\leq 2 n^{1/3} \sum_{i = i^*}^m \frac{1}{\sqrt{g_i \eps n^{1/3}}} + 2 m \nonumber \\
		&\leq \frac{2 n^{1/3}}{\sqrt{g_0 \eps n^{1/3}}} \sum_{i = 0}^\infty {\Big( \frac{4}{5} \Big)}^{\tfrac i4} + O(\log n) = \bigO{\frac{n^{1/3}}{\sqrt{\eps n^{1/3}}}}, \label{eq:goodDif}
	\end{align}
	where we used that $m = O(\log n)$.

	Using that trivially $D_m \leq n$, the inequality in \eqref{eq:goodDif} now gives that
	\begin{align}
		\widehat{\E}[ D_{m} ] & \leq \widehat{\E}[ D_{i^*}1_F ] + n \widehat{\P}(F^c) + \bigO{\frac{n^{1/3}}{\sqrt{\eps n^{1/3}}}} \nonumber \\
		&\leq \widehat{\E}[D_{i^*} 1_F] + n^7 \e^{- (\log n) g_0 /10}  + \bigO{\frac{n^{1/3}}{\sqrt{\eps n^{1/3}}}}. \label{eq:largeg0}
	\end{align}
	Furthermore, on the event $F$ and $i^*=i$, we have $D_i \leq \ell(\mathcal{C}_1(p_{i+1}))$, and hence
	\begin{align*}
		\widehat{\E}[ D_{i^*} 1_F] &= \sum_{i = 0}^m \widehat{\E} \big[ D_{i} 1_F 1_{i^* = i} \big] \\
		&\leq \widehat{\E} \big[ D_0 1_{F} 1_{i^*=0} \big] + \sum_{i=1}^{m-1} \widehat\E \big[ \ell(\mathcal{C}_1(p_{i+1}) ) 1_{i^* = i} \big] + n \P(i^* = m).
	\end{align*}
	If $i^* = i$, then the longest path in $\mathcal{C}_1(p_{i+1})$ satisfies the bound in \eqref{eq:good1} so that
	\begin{align} \label{eq:D-i-star}
		\widehat{\E}[ D_{i^*}1_F ] &\leq  \widehat{\E}[D_0 1_{F} 1_{i^*=0}] + \sum_{i=1}^{m-1} g_{i+1}^4 (\eps n^{1/3})^4 n^{1/3} \P(i^* = i) + n \P(i^* = m)\\
		&\leq \widehat{\E}[ D_0 ] +  {8}  g_0^4 \eps^4 n^{5/3} 
		\sum_{i=1}^{m-1} \exp\Big(-\frac{1}{4} \sqrt{g_{i-1} \eps n^{1/3} }\Big)\left( \frac{5}{4} \right)^{2(i+1)}  \nonumber \\
		& \qquad \qquad \qquad +  {8}  n \e^{- \frac{1}{4} \sqrt{n^{1/3}/\log n}}, \nonumber
	\end{align}
	where we used Lemma \ref{L:good} to upper bound the probabilities 
	\begin{align}
		\P(i^* = i)& \leq \P(i - 1 \text{ is not well-behaved}) \leq \P(A(i-1)^c)+\P(B(i-1)^c) +\P(C(i-1)^c) \nonumber \\
		& \leq  {8}  \exp\Big( -\frac{1}{4} \sqrt{g_{i-1}  \eps n^{1/3}} \Big) \, ,\label{eq:istar}
	\end{align} 
	for $g_0$ large enough.
	Further observe that we have
	\begin{align*}
		&g_0^4 \eps^4 n^{5/3} \sum_{i=1}^{m-1}  \exp\Big(-\frac{1}{4} \sqrt{g_{i-1} \eps n^{1/3} }\Big)\left( \frac{5}{4} \right)^{2(i+1)} 
		\\
		&\hspace{5cm}=O \Bigg( n^{1/3} \big(g_0 \eps n^{1/3}\big)^4 \exp \bigg( -\frac{\sqrt{g_0 \eps n^{1/3}}}{4} \bigg)  \Bigg)
	\end{align*}
	since the first summand dominates the whole sum.
	As $\eps \geq n^{-1/3}$, it follows that 
	\begin{equation*}
		\widehat{\E}[ D_{i^*}1_F ] \leq \widehat{\E}[ D_0 ] + \bigO{\frac{n^{1/3}}{\sqrt{\eps n^{1/3}}}}.
	\end{equation*}
	Possibly increasing $g_0$ to make the exponential in \eqref{eq:largeg0} of order $\bigO{n^{1/3}/ \sqrt{\eps n^{1/3}}}$ finally concludes the proof of the lemma.
\end{proof}

We point out that Remark \ref{R:Critical-pm} follows by bounding 
the term $\widehat{\E}[ D_0 1_{F} 1_{i^*=0}]$ in \eqref{eq:D-i-star} with
\begin{equation*}
	\widehat{\E}[ D_0 1_{F} 1_{i^*=0}]
	\leq \widehat{\E}[ \ell(\cC_1(p_1) 1_{A(1)}]
	\leq g_1^4 n^{1/3} \, ,
\end{equation*}
instead of bounding it by $\widehat \E[D_0]$.



\subsection{Connections of vertices outside the largest component} \label{SS:Connect_outside}

Having obtained a bound on $\widehat{\E}[{ \rm diam}(\ClCp{1}{p_m})]- \widehat{\E}[{ \rm diam}(\ClCp{1}{p_0})]$, we now bound 
$$
\widehat{\E}[{ \rm diam}(\cT)] - \widehat{\E}[{ \rm diam}(\ClCp{1}{p_m})]
$$ 
by showing that, in the random spanning tree $\cT$, the distance between any vertex in $V_n$ and $\ClCp{1}{p_m}$ is uniformly bounded by $(\log n)^c$, for some $c>0$, with high probability. This is established in Lemma \ref{L:fasthitLERW},  
which uses Wilson's algorithm to bound the length of paths in the RSTRE. A key step, stated as Lemma \ref{L:fasthit} below, is to show that the random walk does not visit too many different $p_m$-clusters before hitting the largest component.

In the following, $\omega$ is a given random environment coupled to the random graphs $G_{n,p}$ as in Section~\ref{SS:graph_reduce}, and $\zeta = \zeta(n) > K n^{-1/3} = o(1)$ is chosen such that the theorems in Section~\ref{SS:ERintro} are applicable with $\zeta$ in place of $\varepsilon$.

\begin{lemma} \label{L:fasthit}
	Let $p = \frac{1 + \zeta}{n}$ with $\zeta = \zeta(n) > K n^{-1/3}$ and consider an arbitrary $v_0 \in V_n$. Denote by $Q_{v_0}^\omega$ the probability law of the lazy random walk $(X_t)_{t \geq 0}$ on $(G, \weight)$ started at $X_0 = v_0$, and let $\tau_{\cC_1(p)}$ be the first hitting time of $\mathcal{C}_1(p)$. Define ${\rm Ran}(X)$ as the set of distinct vertices that $(X_t)_{0 \leq t \leq \tau_{\cC_1(p)}}$ visits, and denote by $|{\rm Ran}(X)|$ its cardinality. For all $n$ large enough, and for all $M\in \mathbb{N}$ and $k\geq 2$, we have 
	\begin{equation} \label{eq:fasthit}
		\E \Big[ Q_{v_0}^\omega \Big( |{\rm Ran}(X)| > 10 M(k +  \log(\zeta^3 n)) \zeta^{-2} \Big) \Big] \leq (1 - \zeta)^{M - 1} + 3\e^{-k} + 4 \e^{- \zeta n^{1/3}}.
	\end{equation}
\end{lemma}

Here is a proof sketch. Each time the random walk $X$ visits a new $p$-open cluster other than $\mathcal{C}_1(p)$, it can visit at most $|\mathcal{C}_2(p)|$ new vertices, where $|\cC_2(p)|\leq C \zeta^2 \log(\zeta^3 n)$ with high probability by Theorem \ref{T:size2ndSup}. Hence, we may bound $| {\rm Ran}(X)|$ by $|\mathcal{C}_2(p)|$ times the number of clusters of $G_{n,p}$ other than $\cC_1(p)$ that the random walk visits before time $\tau_{\cC_1(p)}$. Every time the random walk discovers a new cluster there is approximately a $\P$ probability of $|\mathcal{C}_1(p)|/n \approx \zeta$ that the random walk jumps to $\mathcal{C}_1(p)$,  so that the probability of visiting $M$ different clusters before visiting $\mathcal{C}_1(p)$ is roughly $(1-\zeta)^{M-1}$. 

\begin{proof}[Proof of Lemma \ref{L:fasthit}]
	We first restrict to the event that the two largest clusters of $G_{n,p}$ are of typical size. More precisely, denote by $S(\zeta)=S(\zeta,k) := 10 (k +  \log(\zeta^3 n)) \zeta^{-2}$ and let $F_1$ and $F_2$ be the events
	\begin{align*}
		F_1 &:= \Big\{|\mathcal{C}_1(p)| \in \Big[ \frac{3}{2} \zeta n, \frac{5}{2} \zeta n \Big] \Big\}, \\
		F_2 &:= \Big\{|\mathcal{C}_2(p)| \leq S(\zeta) \Big\}.
	\end{align*}
	By Theorem \ref{T:sizeSup} and Theorem \ref{T:size2ndSup}, for all large $n$, we have
	\begin{align}
		&\E  \Big[ Q_{v_0}^\omega \Big(|{\rm Ran}(X)| > M S(\zeta) \Big) \Big] \nonumber \\
		&\hspace{2cm} \leq \E \Big[ Q_{v_0}^\omega \Big(|{\rm Ran}(X)| > M S(\zeta) \Big) 1_{F_1 \cap F_2} \Big] +  \P(F_1^c) + \P(F_2^c) \nonumber \\
		&\hspace{2cm} \leq  \E \Big[ Q_{v_0}^\omega \Big( |{\rm Ran}(X)| > M S(\zeta) \Big)  1_{F_1 \cap F_2} \Big] +  3\e^{-k} + 4 \e^{-\zeta n^{1/3}}  \label{eq:fh1}.
	\end{align}

	For $s \geq 1$ consider the stopping times
	\begin{equation}
		\tau_s := \min \big\{ t \geq 0 : X_t \text{ has visited $s$ distinct  clusters of } G_{n,p}\big\}\,. \label{eq:stoppingk}
	\end{equation}
	We claim that to prove Lemma \ref{L:fasthit}, it suffices to show the following recursive inequality
	\begin{equation} \label{eq:hittingRecursive}
		\E \big[ Q_{v_0}^\omega(\tau_{\cC_1(p)} > \tau_s) 1_{F_1 \cap F_2} \big] \leq (1- \zeta) \E \big[ Q_{v_0}^\omega(\tau_{\cC_1(p)} > \tau_{s-1}) 1_{F_1 \cap F_2} \big].
	\end{equation}
	On the event $F_2$ each cluster $\mathcal{C}_\ell(p)$, $\ell \geq 2$, has size at most $S(\zeta)$. Since $|{\rm Ran}(X)|$ counts each vertex at most once, if $|{\rm Ran}(X)| > MS(\zeta)$, then the random walk $X$ must have visited at least $M$ distinct clusters other than $\cC_1(p)$. In particular,
	\begin{align*}
		\E \Big[ Q_{v_0}^\omega \Big( |{\rm Ran}(X)| > M S(\zeta) \Big) 1_{F_1 \cap F_2} \Big] &
		\leq \E \big[ Q_{v_0}^\omega( \tau_{\cC_1(p)} > \tau_M) 1_{F_1 \cap F_2} \big] \\
		&\leq (1- \zeta)^{M - 1},
	\end{align*}
	which together with \eqref{eq:fh1} completes the proof of \eqref{eq:fasthit}.

	We now proceed to prove the recursive inequality \eqref{eq:hittingRecursive}. Recall that $\weight(u):=\sum_{x\sim u} \weight(u, x)$. For a given environment $\omega$ and for a fixed random walk trajectory $\Vec{x} = (x_0, \ldots, x_{m-1})$, consider the $\sigma$-algebras
	\begin{align*}
		\mathcal{F}_p &:= \sigma \big( (1_{\omega_e \leq p})_{e \in K_n} \big) = \sigma(G_{n,p}), \\
		\mathcal{F}_{\Vec{x}} &:= \sigma \Big( \big( \weight(u) \big)_{u \in \Vec{x}}, \big( \weight(e) \big)_{e \in \Vec{x}} \Big),
	\end{align*}
	where $u \in \Vec{x}$ and $e \in \Vec{x}$ denote, respectively, that $u$ is one of the vertices in $\Vec{x}$ and $e$ is one of the edges $(x_{i-1}, x_{i})$ for $i=1, \ldots, m-1$. 
	For a stopping time $\tau$ of the random walk $X$, we write $\tau(\Vec{x}) \leq t$ if conditioned on $\{(X_0, \ldots X_{m-1}) = \Vec{x}\}$, the event $\{\tau \leq t\}$ holds almost surely w.r.t.\ to the law $Q^\omega_{v_0}$ of the random walk, and define $\tau(\Vec{x}) = t$ and $\tau(\Vec{x}) > t$ analogously. Lastly, we denote by $(\Vec{x}, z) := (x_0, \ldots, x_{m-1}, z)$ the concatenation of $\Vec{x}$ and $z$, and by
	\begin{equation*}
		\mathcal C_{\Vec x}(p) := \big\{ v \in V_n \, : \, \exists x_i \in \Vec{x},\ v \xleftrightarrow{p} x_i \big\} \subseteq V_n,
	\end{equation*}
	the union of clusters in $G_{n,p}$ visited by the trajectory $\Vec x$.
	
	For a given $s>0$, let $\Xi_m=\Xi_m(s)$ be the set of paths $\Vec{x} = (x_0, \ldots, x_{m-1})$ such that $\tau_{s-1}(\Vec{x}) \leq m-1$, $\tau_{s}(\Vec{x}) > m-1$ and $x_i \not\in \Conep$ for all $0 \leq i \leq m-1$, i.e.\ the set of paths which have visited exactly $s-1$ clusters by time $m-1$, none of which is $\Conep$. Fix some $\Vec{x} \in \Xi_m$ and condition on the event $\{ (X_0, \ldots X_{m-1}) = \Vec{x}\}$ and the $\sigma$-algebras $\mathcal{F}_p, \mathcal{F}_{\Vec{x}}$. We wish to compare the probabilities of random walk trajectories that visit at time $m$ a $s$-th new cluster different from $\cC_1(p)$, versus those trajectories where the $s$-th new cluster could also be $\cC_1(p)$. The probability of the first set of trajectories satisfies 
	\begin{align}
		&\E \Big[ Q_{v_0}^\omega \Big( \tau_{\cC_1(p)} > \tau_s = m \mid (X_0, \ldots, X_{m-1})
		= \Vec{x} \Big) \ \Big| \ \sigma(\mathcal{F}_{p}, \mathcal{F}_{\Vec{x}}) \Big] \nonumber \\
		&\quad =\E \Big[ \sum_{z \in V_n} 1_{\tau_s(\Vec{x}, z) = m < \tau_{\cC_1(p)}(\Vec{x}, z)} \  Q_{v_0}^\omega \Big( X_{m} = z \mid X_{m-1} = x_{m-1} \Big) \ \Big| \ \sigma(\mathcal{F}_{p}, \mathcal{F}_{\Vec{x}}) \Big] \nonumber \\
		&\quad = \frac{1}{\weight(x_{m-1})} \E \Big[ \sum_{z \in V_n \setminus (\mathcal C_{\Vec x}(p) \cup \mathcal{C}_1(p)) } \weight(x_{m-1}, z) \ \Big| \ \sigma(\mathcal{F}_{p}, \mathcal{F}_{\Vec{x}}) \Big] =: \lambda_1(m, \Vec{x})\,, \label{eq:lambda1}
	\end{align}
	where in the last line we have used that $\weight(x_{m-1})$ is $\mathcal{F}_{\Vec{x}}\,$-measurable.
	Similarly, for the probability that the random walk visits at time $m$ a $s$-th new cluster including $\cC_1(p)$, we have
	\begin{align}
		&\E \Big[ Q_{v_0}^\omega \Big( \tau_{\cC_1(p)} \geq \tau_s = m \mid (X_0, \ldots, X_{m-1})
		= \Vec{x} \Big) \ \Big| \ \sigma(\mathcal{F}_{p}, \mathcal{F}_{\Vec{x}}) \Big] \nonumber \\
		&\quad =\E \Big[ \sum_{z \in V} 1_{\tau_s(\Vec{x}, z) = m \leq \tau_{\cC_1(p)}(\Vec{x}, z)} \  Q_{v_0}^\omega \Big( X_{m} = z \mid X_{m-1} = x_{m-1} \Big) \ \Big| \ \sigma(\mathcal{F}_{p}, \mathcal{F}_{\Vec{x}}) \Big] \nonumber \\
		&\quad = \frac{1}{\weight(x_{m-1})} \E \Big[ \sum_{z \in V_n \setminus \mathcal C_{\Vec x}(p)} \weight(x_{m-1}, z) \ \Big| \ \sigma(\mathcal{F}_{p}, \mathcal{F}_{\Vec{x}}) \Big] =: \lambda_2(m, \Vec{x}). \label{eq:lambda2}
	\end{align}
	
	Conditioned on $\mathcal{F}_p$ and $\mathcal{F}_{\Vec{x}}$,  the edge weights $(\weight(x_{m-1}, z))_{z\in V_n\backslash \mathcal C_{\Vec x}(p)}$ are not independent since we have conditioned on $\weight(x_{m-1}) = \sum_{z\sim x_{m-1}} \weight(x_{m-1}, z) \in \mathcal{F}_{\vec x}$, but their joint law is exchangeable. In particular, 
	\begin{align}
		\frac{\lambda_1(m, \Vec{x})}{\lambda_2(m, \Vec{x})} &= \frac{|V_n \setminus (\mathcal C_{\Vec x}(p) \cup \mathcal{C}_1(p))|}{|V_n \setminus \mathcal C_{\Vec x}(p)|} \nonumber \\
		&= \frac{n - |\mathcal C_{\Vec x}(p)| - |\mathcal{C}_1(p)|}{n - |\mathcal C_{\Vec x}(p)|} \leq 1 - \frac{|\mathcal{C}_1(p)|}{n}. \label{eq:ratiolambda}
	\end{align}
	
	We perform a decomposition of the random walk trajectory up to time $\tau_s = m$, for all possible values of $m$ and all possible trajectories in $\Xi_m$. By first conditioning on $\sigma(\mathcal{F}_{\Vec{x}} ,\mathcal{F}_p)$ and then on $\mathcal{F}_p$, we have
	\begin{align*}
		&\E[Q_{v_0}^\omega(\tau_{\cC_1(p)} > \tau_s) 1_{F_1 \cap F_2}] \\
		&\hspace{1cm}=         \sum_{ m \geq 1} \sum_{ \Vec{x} \in \Xi_m}  \E\big[ Q_{v_0}^\omega(\tau_{\cC_1(p)} > \tau_s = m, \, (X_0, \ldots, X_{m-1}) = \Vec{x}) \big]\\
		&\hspace{1cm}= \sum_{ m \geq 1} \sum_{ \Vec{x} \in \Xi_m} \E \bigg[ \E \Big[ \E\big[ Q_{v_0}^\omega(\tau_{\cC_1(p)} > \tau_s = m \mid (X_0, \ldots, X_{m-1}) = \Vec{x}) \ \big| \ \sigma(\mathcal{F}_{p}, \mathcal{F}_{\Vec{x}}) \big] \\
		& \hspace{7cm} \cdot Q_{v_0}^\omega ((X_0, \ldots, X_{m-1}) = \Vec{x}) \ \Big| \ \mathcal{F}_p \Big] 1_{F_1 \cap F_2} \bigg],
	\end{align*}
	where we used that $Q_{v_0}^\omega( (X_0, \ldots, X_{m-1}) = \Vec{x})$ is $\sigma(\mathcal{F}_{p}, \mathcal{F}_{\Vec{x}})$ measurable. Using the definition of $\lambda_1(m, \Vec{x})$ in \eqref{eq:lambda1}, we see that the above expression is equal to 
	\begin{equation}
		\sum_{ m \geq 1} \sum_{ \Vec{x} \in \Xi_m} \E \Big[ \E \Big[ \frac{\lambda_1(m, \Vec{x})}{\lambda_2(m, \Vec{x})}\lambda_2(m, \Vec{x}) Q_{v_0}^\omega( (X_0, \ldots, X_{m-1}) = \Vec{x}) \ \Big| \ \mathcal{F}_p \Big]  1_{F_1 \cap F_2} \Big]\,, \label{eq:decomp_ratiol1l2} 
	\end{equation}
	which is bounded above by
	\begin{equation*}
		\sum_{ m \geq 1} \sum_{ \Vec{x} \in \Xi_m}   \E \Big[ \E \Big[ \lambda_2(m, \Vec{x}) Q_{v_0}^\omega( (X_0, \ldots, X_{m-1}) = \Vec{x}) \ \Big| \ \mathcal{F}_p \Big] \Big( 1 - \frac{|\mathcal{C}_1(p)|}{n} \Big) 1_{F_1 \cap F_2} \Big]   \, . \nonumber 
	\end{equation*}
	Here we have used inequality \eqref{eq:ratiolambda} and the fact that $|\mathcal{C}_1(p)|$ is $\mathcal{F}_p$ measurable. Since by \eqref{eq:lambda2}
	\begin{equation*}
		\sum_{ m \geq 1} \sum_{ \Vec{x} \in \Xi_m} \E \Big[ \E \big[ \lambda_2(m, \Vec{x}) Q_{v_0}^\omega( (X_0, \ldots, X_{m-1}) = \Vec{x}) \ \big| \ \mathcal{F}_p \big] \Big] = \E\big[Q_{v_0}^\omega(\tau_{\cC_1(p)} \geq \tau_{s}) \big]\,,
	\end{equation*} 
	we have that
	\begin{align*}
		\E[Q_{v_0}^\omega(\tau_{\cC_1(p)} > \tau_s) 1_{F_1 \cap F_2}] &\leq \E \Big[ Q_{v_0}^\omega(\tau_{\cC_1(p)} \geq \tau_{s}) \Big (1 - \frac{|\Conep|}{n} \Big) 1_{F_1 \cap F_2} \Big] \\
		&\leq (1- \zeta) \E \big[ Q_{v_0}^\omega(\tau_{\cC_1(p)} > \tau_{s-1}) 1_{F_1 \cap F_2} \big]\,,
	\end{align*}
	where the last inequality used that $|\Conep| \geq \zeta n$ on the event $F_1$.
\end{proof}

As a consequence of Lemma \ref{L:fasthit}, the length of the LERW used in Wilson's algorithm to connect any vertex $v_0$ to $\ClCp{1}{p_m}$ can not become too large. This length equals  $d_{\cT} \big( v_0, \ClCp{1}{p_m} \big)$, the distance between $v_0$ and $\ClCp{1}{p_m} $ in the spanning tree $\cT$. We state this in the following Lemma. 
\begin{lemma} \label{L:fasthitLERW}
	Let $\cT$ be the RSTRE on the complete graph $(K_n,\weight_n)$ with law $\bP^\omega_{n, \beta_n}$ as in Proposition \ref{P:DiaC1}. 
	Let $p_m$ be defined as in \eqref{eq:g_seq} and \eqref{eq:definitionm}. Then for all $v_0 \in V_n$
	\begin{equation*}
		\widehat{\P} \Big( d_{\cT} \big( v_0, \ClCp{1}{p_m} \big) \geq 
		(\log n)^6\Big) = \bigO{ \frac{1}{n^2} }\,.
	\end{equation*}
\end{lemma}
\begin{proof}
	To construct the path in $\cT$ from $v_0$ to $\ClCp{1}{p_m}$, we can first run Wilson's algorithm on $(K_n ,\weight_n)$ to generate $\ClCp{1}{p_m}$ by starting with an arbitrary vertex in $\cC_1(p_m)$ and then running successive LERWs from the remaining vertices of $\cC_1(p_m)$. 
Then we can run a loop erased random walk from $v_0$ to $\ClCp{1}{p_m}$. 

We will apply Lemma \ref{L:fasthit} with $\zeta = g_m \eps$ (and $p_m = (1 + \zeta)/n$), where we remark that, by definition of $m$ in \eqref{eq:definitionm}, we have $(\log n)^{-1} \leq \zeta =o(1)$. Consider a random walk $X$ on the weighted graph $(K_n,\weight_n)$ started at $v_0$ and stopped when it reaches $\mathcal C_1(p_m)$. 
Let $Y$ be the loop erasure of $X$ until the first time $X$ hits $\ClCp{1}{p_m}$.
Clearly the length of the trajectory of $X$ dominates that of the trajectory of $Y$.

Thus for large enough $n$, by Lemma \ref{L:fasthit} with $k= 2 \log n$ and $M =  \lceil 3 (\log n)^2 \rceil + 1$ together with the facts that $\zeta^{-1}\leq \log n$ and 
\begin{equation*}
    10 M (2 \log n + \log(\zeta^3 n)) \zeta^{-2} \leq (\log n)^6,
\end{equation*}
we have
\begin{align*}
	\widehat{\P}\Big( d_{\cT}\big(v_0, \ClCp{1}{p_m}\big)\geq (\log n)^6 
	\Big) 
	&\leq  \E \Big[ Q_{v_0}^\omega \Big( |{\rm Ran}(X)| >  10 M (2 \log n + \log(\zeta^3 n)) \zeta^{-2} \Big) \Big]\\
	&\leq \big( 1 - \frac{1}{\log n} \big)^{3 (\log n)^2} + 3\e^{-2\log n} + 4 \e^{-n^{1/4}} \\
	&= \bigO{\frac{1}{n^2}}. 
\end{align*} 
\end{proof}

\subsection{Proof of Proposition \ref{P:DiaC1}} \label{SS:ProofHigh}

We now show how Lemma \ref{L:sizepm} and Lemma \ref{L:fasthit} can be combined to prove Proposition \ref{P:DiaC1}.

\begin{proof}[Proof of Proposition \ref{P:DiaC1}]
Recall the sequence    $(p_i)_{0 \leq i \leq m}$     from \eqref{eq:g_seq}. The lower bounds on the diameter of $\cT$ in \eqref{eq:E_DiaC1} and \eqref{eq:P_DiaC1} hold trivially.   We proceed to prove the upper bounds.

By Lemma \ref{L:sizepm}, we obtain that
\begin{align} \label{eq:mainproofeq}
	\widehat{\E}[ {\rm diam}(\cT)] &\leq \widehat{\E} \big[ {\rm diam} \big( \ClCp{1}{p_m} \big) \big] +  2\widehat{\E} \big[ \max_{v \in V_n} d_\cT \big( v,\ClCp{1}{p_m} \big) \big] \nonumber \\
	&\leq \widehat{\E} \big[ {\rm diam} \big( \ClCp{1}{p_0} \big) \big] +  2\widehat{\E}\big[ \max_{v \in V_n} d_\cT \big( v,\ClCp{1}{p_m} \big) \big] + \bigO{\frac{n^{1/3}}{\sqrt{\eps n^{1/3}}}}.
\end{align}
Consider the event
\begin{equation}
	\text{Bad} := \big\{ \exists v \in V_n \text{ such that } d_\cT(v,\ClCp{1}{p_m}) \geq (\log n)^6
	\big\}.
\end{equation}
As the length of the longest acyclic path can be trivially bounded by $n$, we have
\begin{equation}
	\widehat{\E} \big[ \max_{v \in V_n} d_\cT \big( v,\ClCp{1}{p_m} \big) \big] \leq (\log n)^6
	+ n \widehat{\P}(\text{Bad}). \label{eq:diamwithbad}
\end{equation}
By a union bound and Lemma \ref{L:fasthitLERW} (which holds uniformly for $v \in V_n$)
\begin{equation}
	\widehat{\P}(\text{Bad}) \leq n \widehat{\P}\Big(d_\cT \big( v_0,\ClCp{1}{p_m} \big) \geq (\log n)^6
	\Big) 
	= \bigO{\frac{1}{n}}. \label{eq:Bad_bound}
\end{equation}
Inserting this bound into \eqref{eq:diamwithbad} and subsequently \eqref{eq:mainproofeq} completes the proof of the equation \eqref{eq:E_DiaC1}. We are left to show that the upper bound in \eqref{eq:P_DiaC1} holds with probability high enough.

Consider the event
\begin{equation*}
	\mathcal{B} = F \cap \{ i^* = 0 \} \cap  \text{Bad}^c,
\end{equation*}
where $i^*$ and $F$ are those of \eqref{eq:i*} and \eqref{eq:F_subset} in the proof of Lemma \ref{L:sizepm}. On the event $\mathcal{B}$, we have as in the argument around \eqref{eq:goodDif} that
\begin{align*}
	{\rm diam}(\cT) &\leq {\rm diam} \big( \ClCp{1}{p_m} \big) +  2 \max_{v \in V_n} d_\cT \big( v,\ClCp{1}{p_m} \big) \\
	&\leq {\rm diam} \big( \ClCp{1}{p_0} \big) + \bigO{\frac{n^{1/3}}{\sqrt{\eps n^{1/3}}}} +O((\log n)^6),
\end{align*}
so it suffices to show that $\P(\mathcal{B}^c)$ is small.

Using the probability bounds on $F^c$ , $\{i^*=i\}$ and $\text{Bad}$ from \eqref{eq:F_subset_bound}, \eqref{eq:istar} and \eqref{eq:Bad_bound}, respectively, we have     
\begin{align*}
	\widehat{\P}(\mathcal{B}^c) &\leq  \widehat{\P}(F^c) + \sum_{i=1}^m \P(i^* = i) + \widehat{\P}(\text{Bad}) \\
	&\leq n^6 \e^{-\frac{g_0 \log n }{10}} + \bigO{ \exp \Big(-\frac{1}{4} \sqrt{g_0 \eps n^{1/3}} \Big)} + \bigO{ \frac{1}{n} },
\end{align*}
which for large enough $g_0$ satisfies the probability bound of Proposition \ref{P:DiaC1}.
\end{proof}

\subsection{Proof of Theorem \ref{T:highDisorder}} \label{SS:Thm_high}The proof of Theorem \ref{T:highDisorder} uses Proposition \ref{P:DiaC1} with $\eps = n^{-1/3}$. In this case, $p_0$ is in the critical window and the largest component is ``almost'' a tree, typically having a bounded number of excess edges (recall that the excess of a graph $G = (V,E)$ was defined as ${\rm Exc }(G) := |E| - |V|$). It follows that every spanning tree of $\mathcal{C}_1(p_0)$ has essentially the same diameter as that of $\mathcal{C}_1(p_0)$ itself, which is of order $n^{1/3}$. Before giving the proof details, we prove the following auxiliary lemma.

\begin{lemma} \label{L:ratiodT1T2}
Let $T_1, T_2$ be two spanning trees of a connected graph $G$. If the edge sets $E(T_1)$ of $T_1$ and $E(T_2)$ of $T_2$ satisfy $|E(T_2) \setminus E(T_1)| = k$, then
\begin{equation}
	\frac{{\rm diam}(T_2)}{k+1} - 1 \leq {\rm diam}(T_1) \leq (k+1) {\rm diam}(T_2) + k  \,. \label{eq:ratio_dT1T2}
\end{equation}
\end{lemma}
\begin{proof}
Note that $|E(T_1)| = |E(T_2)|$. We can construct $T_1$ from $T_2$ by first removing the $k$ edges in $E(T_2)\backslash E(T_1)$ from $T_2$ to obtain some forest $F \subseteq G$ with $(k+1)$ connected components, and then adding $k$ edges to $F$ to obtain $T_1$. Let $d_F$ be the maximum diameter among all connected components of $F$. It follows that
\begin{align*}
	d_F \leq {\rm diam}(T_1) &\leq (k+1) d_F + k\,, \\
	d_F \leq {\rm diam}(T_2) &\leq (k+1) d_F + k\,.
\end{align*}
Rearranging the above inequalities readily gives \eqref{eq:ratio_dT1T2}.
\end{proof}

In order to bound $\widehat{\P}( {\rm diam}(\cT) \geq C^{-1} n^{1/3})$ in Theorem \ref{T:highDisorder}, we will lower bound ${\rm diam}(\cT)$ by ${\rm diam}(\cT_{\cC_1(p_0)})$. The tree $\cT_{\cC_1(p_0)}$ is contained in a connected component of $G_{n, p_1}$, and the induced subgraph $H$ of $G_{n, p_1}$ with the same vertex set as $\cT_{\cC_1(p_0)}$ is a graph with small excess.  On $H$, we can construct a spanning tree $T$ with diameter that is bounded from below by the diameter of $\cC_1(p_0)$, which is known to be of order $n^{1/3}$ with high probability. We can then apply Lemma \ref{L:ratiodT1T2} to $\cT_{\cC_1(p_0)}$ and $T$ to obtain a lower bound on the diameter of $\cT_{\cC_1(p_0)}$.

\begin{proof}[Proof of Theorem \ref{T:highDisorder}]
Let $\beta_n \geq n^{4/3} \log n$, $\eps = n^{-1/3}$ and fix $\delta > 0$ as in the statement of Theorem \ref{T:highDisorder}. Consider the constant $g_0$ from Proposition \ref{P:DiaC1} with $p_0 = (1 + g_0 \eps)/n$. As $\beta_n \eps \geq n \log n$, Proposition \ref{P:DiaC1} and Remark \ref{R:Critical-pm} give that 
\begin{equation*}
	\widehat{\E}[{\rm diam}(\cT)] = O(n^{1/3})\,.
\end{equation*}
By Markov's inequality, for large enough $C_1 = C_1(\delta)$ one immediately has
\begin{equation*}
	\widehat{\P} \big( {\rm diam}(\cT) > C_1 n^{1/3} \big) \leq \frac{\delta}{2}\,.
\end{equation*}
To prove Theorem \ref{T:highDisorder}, it thus suffices to show that
\begin{equation}
	\widehat{\P} \big( {\rm diam}(\cT) < C_1^{-1} n^{1/3} \big) 
	\leq \frac{\delta}{2}\,, \label{eq:HighDis_lower_bound}
\end{equation}
as this also implies that $\widehat{\E}[{\rm diam}(\cT)] \geq C n^{1/3}$ for some other constant $C > 0$.

Recall that the excess of a graph was defined as ${\rm Exc }(G) = |E(G)| - |V(G)|$ and that we defined $p_1 = (1 + (5/4)^{1/2}g_0 \eps)/n$. Consider the following events:
\begin{align*}
	D &:= \big\{ {\rm diam}(\mathcal{C}_1(p_0)) \geq A^{-1} n^{1/3} \big\}, \\
	F(0) &:= \big\{ \ClCp{1}{p_0} \subseteq G_{n, p_1} \big\},  \\
	C(0) & := \big\{ \mathcal{C}_1(p_0) \subseteq \mathcal{C}_1(p_1) \big\}, \\
	Q &:= \big\{  {\rm Exc }(\mathcal{C}_1(p_1)) \leq 200 g_0^3 \big\},
\end{align*}
where $A=A(\delta/4, g_0)$ is the constant appearing in Theorem \ref{T:diamCriticalp}. We have the following probability bounds:
\begin{center}
	\begin{tabular}{ l r }
		$\P(D^c) \leq \frac{\delta}{4}$ & (by Theorem \ref{T:diamCriticalp}), \\
		$\widehat{\P}(F(0)^c) \leq n^5 \exp\Big(- \frac{g_0 \log n}{10}\Big)$ & (by Lemma \ref{L:gap}), \\ 
		$\widehat{\P}(C(0)^c) \leq 2 \exp\Big(- \frac{1}{4} \sqrt{g_0}\Big)$ & (by Lemma \ref{L:good}), \\ 
		$\P(Q^c) \leq 2 \e^{-g_0}$ & (by Theorem \ref{T:sizeSup}). \\ 
	\end{tabular}
\end{center}
Thus for large enough $g_0 = g_0(\delta)$, we obtain
\begin{equation} \label{eq:combine_event_bound}
	\widehat{\P} \Big(D \cap F(0) \cap C(0) \cap Q \Big) \geq 1 - \frac{\delta}{2},
\end{equation}
and for the remainder of the proof we will restrict to all of these events.

Denote by $H$ the subgraph of $G_{n, p_1}$ induced by the vertices of $\ClCp{1}{p_0}$, i.e.\ $H$ contains all the vertices in $\ClCp{1}{p_0}$ and all $p_1$-open edges whose endpoints are both in $\ClCp{1}{p_0}$. On the event $F(0) \cap C(0)$ we have that 
\begin{gather*}
	\ClCp{1}{p_0} \subseteq H \subseteq \mathcal{C}_1(p_1),
\end{gather*}
and furthermore, on the event $Q$ we also see that
\begin{equation}
	{\rm Exc}(H) \leq {\rm Exc}(\mathcal{C}_1(p_1)) \leq 200 g_0^3. \label{eq:excessH}
\end{equation}

Let $T^0$ be any spanning tree of $\mathcal{C}_1(p_0)$ and (arbitrarily) extend $T^0$ to a spanning tree $T$ on $H$. On the event $D$ we have the inequalities
\begin{equation*}
	{\rm diam}(T) 
	\geq {\rm diam}(T^0) 
	\geq {\rm diam}(\cC_1(p_0)) \geq A^{-1} n^{1/3}.
\end{equation*}
Both $T$ and $\ClCp{1}{p_0}$ are spanning trees of $H$ so that both trees select exactly $|V(H)| - 1$ edges from a choice of $|E(H)| = {\rm Exc}(H) + |V(H)|$ edges. In particular, we have
\begin{equation*}
	\big| E(T) \setminus E \big( \ClCp{1}{p_0} \big) \big| \leq {\rm Exc}(H) + 1,
\end{equation*}
i.e.\ $T$ and $\ClCp{1}{p_0}$ have (on the event $Q$) many edges in common. By \eqref{eq:excessH} and Lemma \ref{L:ratiodT1T2} we obtain that
\begin{equation*}
	{\rm diam}(\cT) \geq {\rm diam}\big( \ClCp{1}{p_0} \big) \geq \frac{n^{1/3}}{A ( 200g_0^3 + 2) } -1 \geq C_1^{-1}n^{1/3}.
\end{equation*}
Provided that $C_1$ is large enough, this together with \eqref{eq:combine_event_bound} proves \eqref{eq:HighDis_lower_bound} 
\end{proof}

\smallskip

\begin{remark} \label{R:MSTdiff}
Our proof of Proposition \ref{P:DiaC1} adapts arguments from the proof of \cite[Theorem 1]{ABR09} on the diameter of the minimum spanning tree. In particular, we follow their strategy of controlling the diameter of the spanning tree on the whole graph $(K_n, \weight_n)$ in terms of the diameter of the spanning tree restricted to the largest component of a critical Erdős-R\'enyi random graph coupled to the random environment $\omega$. However, there are some key differences between the MST, which corresponds to $\beta=\infty$, and our case with $n^{4/3} \leq \beta_n<\infty$: 
\begin{enumerate}
	\item When $\omega$ is coupled to $G_{n, p_i}$, 
	the minimum spanning tree $\cT^{\min}$ on $(K_n, \omega)$ trivially satisfies the property that  $\cT_{\cC_1(p_i)}^{\min}$ (see definition in \eqref{eq:defCbar}) is contained in $\cC_1(p_i)$. This does not hold for the RSTRE with $\beta_n<\infty$, which motivates the definition of $\cT_{\cC_1(p_i)}$. The key ingredient of our analysis is Lemma \ref{L:gap}, which implies that, with high probability, $\cT_{\cC_1(p_i)} \subseteq \cC_1(p_{i+1})$ instead of $\cT_{\cC_1(p_i)} \subseteq \cC_1(p_i)$. Further work needs to be done to compare the diameter of $\cT_{\cC_1(p_i)}$ with that of $\cT_{\cC_1(p_{i+1})}$ that is specific to our setting.

	\item The minimum spanning tree $\cT^{\min}$ on $(K_n, \omega)$ can be constructed through a greedy algorithm (Prim's algorithm) that only uses local information as we grow the tree.    	
	For the RSTRE on $(K_n, \weight_n)$, we apply Wilson's algorithm, which uses global information of the weighted graph due to the loop erasure: the LERW depends on more than just the random environment observed along the LERW. 
	This complication arises in particular in the proof of Lemma \ref{L:fasthit} and Lemma \ref{L:fasthitLERW}. We deal with this by working under the averaged measure $\widehat{\E}$ and conditioning on previously gained information of the random walk.
	
	\item 
	In the analogue of Proposition \ref{P:DiaC1} in \cite{ABR09} the authors only consider the starting parameter $p_0$ inside the critical window, which corresponds in our case to $\beta_n\geq n^{4/3}\log n$. When $\beta_n<n^{4/3}\log n$, we are forced to choose $p_0$ slightly above the critical window: this led to  Conjecture \ref{C:Intermediate}, see Section \ref{S:Conclusion}.  
\end{enumerate}
\end{remark}

\subsection{Very high disorder} \label{SS:veryHigh}

We end this section by providing an alternative approach to proving Theorem \ref{T:main} in the high disorder regime when $\beta_n$ grows fast enough. Namely, for any graph sequence $G_n$, we will show that for large enough $\beta_n$ the law of the RSTRE concentrates on a single tree, the MST. In particular, when $G_n=K_n$ are complete graphs, by the results of \cite{ABR09}, the MST has a diameter of order $n^{1/3}$. 
\begin{theorem} \label{T:SuperHighDisorder}
Let $G_n = (V_n, E_n)$ be a sequence of connected graphs, with $|V_n| = n$ and $|E_n| = m(n) = m$. Assign to each edge a weight $\weight_n(e) = \exp(-\beta_n \omega_e)$, where $(\omega_e)_{e \in E_n}$ are i.i.d.\ uniform r.v.'s on $[0,1]$. Let $\cT^{\min}(n)$ be the spanning tree that  minimizes $H(\cT, \omega):=\sum_{e\in \cT}\omega_e$ in $\mathbb{T}_{G_n}$, the set of all spanning trees of $G_n$.  If $\beta_n \gg m^2 n \log n$, then 
\begin{align*}
	\widehat{\P} \Big(\cT_{n,\beta_n}^\omega=\cT^{\min}(n) \Big)\xrightarrow{n\to\infty}1\,.
\end{align*}

\end{theorem}
\begin{proof}[Proof Sketch]
We follow the same proof steps as those of \cite[Proposition 6.1]{MSS23}. Order the spanning trees of $G_n$ by their weights $\weight_n(T)=\prod_{e \in T} \weight_n(e)$ and consider $T_1 =\cT^{\min}(n)$ and $T_2$, the two trees with the largest weights (i.e.\ smallest $H(T, \omega)$). It suffices to prove that if $\beta_n$ is large enough, then
\begin{equation*} 
	\big| \mathbb T_{G_n} \big|\frac{ \weight_n(T_2)}{\weight_n(T_1)}  = \big|\mathbb T_{G_n} \big| \exp\Big( -\beta_n \big( \sum_{e \in T_2} \omega_e - \sum_{e \in T_1} \omega_e \big) \Big)  \xrightarrow{n \rightarrow \infty} 0 \, ,
\end{equation*}
where $\big| \mathbb T_{G_n} \big|$ denotes the cardinality of the set of all possible spanning trees of $G_n$.
As $\big| \mathbb T_{G_n} \big|\leq n^{n-2}$ for any simple graph with $n$ vertices (Cayley's theorem), and since $T_1$ and $T_2$ differ by exactly one edge, we have that
\begin{equation} \label{eq:T1vsT2}
	\big|\mathbb T_{G_n} \big| \exp\Big( -\beta_n \Big( \sum_{e \in T_2} \omega_e - \sum_{e \in T_1} \omega_e \Big) \Big) 
	\leq \exp\Big((n-2) \log n-\beta_n \min\limits_{\substack{e,e' \in E_n \\ e\neq e'}} |\omega_e - \omega_{e'}| \Big)\,.
\end{equation}
The minimum difference between $k$ independent uniform $[0,1]$-valued random variables is typically of order $1/k^2$ (see the proof of \cite[Proposition 6.1]{MSS23}), so that if $\beta_n \gg m^2 n \log n$, then \eqref{eq:T1vsT2} indeed goes to zero.
\end{proof}



\section{Heuristics for Conjecture \ref{C:Intermediate}} \label{S:Conclusion}

We provide here some heuristic justification for Conjecture \ref{C:Intermediate} on the diameter of the RSTRE in the intermediate disorder regime $\beta_n = n^{1 + \gamma}$ for $0 < \gamma < 1/3$, and refer to \cite{Mak25} for a more detailed explanation. As we have seen in Proposition \ref{P:DiaC1}, 
\begin{equation*}
{\rm diam}(\cT) \approx {\rm diam}\big( \ClCp{1}{p_0} \big)
\end{equation*}
for $p_0=   \tfrac{1}{n}+ \tfrac{c \log n}{\beta_n}$. 
Using Lemma \ref{L:gap}, we may therefore restrict ourselves to studying the RSTRE on $\cC_1(p)$ for $p>p_0$ in the slightly supercritical window.

\smallskip

For simplicity assume that $\beta_n = n^{1 + \gamma}$ with $1/4 < \gamma <1/3$, and let $\eps = \tfrac{C n\log n}{\beta_n}$ and $p = (1+\eps)/n$. In \cite{DKLP11}, the authors construct a graph $\tilde{\cC}_1(p)$ that is \textit{contiguous} with $\cC_1(p)$, that is, every graph property that holds with high probability for $\tilde{\cC}_1(p)$ also holds with high probability for $\mathcal C_1(p)$. The construction of $\tilde{\mathcal C}_1(p)$ consists of three steps:
\begin{enumerate}
\item Sample a random 3-regular graph with $\Theta(\eps^3 n)$ vertices to generate the so called \textit{kernel graph} $K$; \label{enu:kernel}
\item Replace each edge of $K$ independently by a path with length following the geometric distribution Geom$(\eps)$, to obtain the so called \textit{2-core} $H$; \label{enu:2core}
\item Attach to each vertex of $H$ an independent Galton-Watson tree with offspring distribution following the Poisson distribution Poisson$(1-\eps)$. \label{enu:GWtree}
\end{enumerate}
The attached trees in step (\ref{enu:GWtree}) can only increase the diameter of the random spanning tree $\cT$ on $\tilde{\mathcal{C}}_1(p)$ by $O(\log n)$, so that we may safely ignore them. To each edge in the 2-core $H$ we assign weights of the form $\weight_n(e) = \exp(-\beta_n p \, \tilde\omega_e)$, where $\tilde\omega_e$ are i.i.d.\ uniform on $[0,1]$, corresponding to the original weights conditioned on the events $\{\omega_e\leq p\}$. Using the series law, this also induces i.i.d.\ weights on the edges of the kernel graph $K$. We can couple the RSTREs on $H$ and $K$ by keeping an edge in $K$ if and only if the whole corresponding path in $H$ is contained in the RSTRE. The diameter of the RSTRE on $\tilde{\mathcal{C}}_1(p)$ can then be obtained 
by multiplying the diameter of the RSTRE on $K$ by $\eps^{-1}$. Hence, if we denote by $\cT^H_p$ and $\cT^K_p$ the RSTRE's on $H$ and $K$, respectively, then
\begin{equation}
{\rm diam}(\cT) \approx  {\rm diam}(\cT^H_p) \approx \frac{\beta_n}{n}{\rm diam}(\cT^K_p),
\end{equation} 
where $\approx$ ignores factors of order $(\log n)^c$ for $c > 0$.

\smallskip

In \cite{MSS23}, it was shown that for bounded degree expanders $G$, with edge weights $\weight(e)$ whose distribution does not change with $G$, 
the RSTRE has a diameter of order $\sqrt{|V(G)|}$ with high probability. The kernel graph $K$ satisfies most of the conditions of \cite[Theorem 1.1]{MSS23}, except that the weight distribution now depends on $n$. However, the law of these edge weights are more concentrated because the resistance on each edge of the kernel graph $K$ is the sum of ${\rm Geom}(\eps)$ i.i.d.\ random variables of the form $\exp( \beta_n p \, \tilde \omega_e)$. It is then conceivable that the result of \cite{MSS23} still holds for such weighted kernel graphs. Since the size of the kernel graph $K$ 
is of order $n^4 (\log n)^3/ \beta_n^3$, this suggests that
\begin{equation}
{\rm diam}(\cT)  \approx \frac{\beta_n}{n} \sqrt{\frac{n^4 }{\beta_n^3}}  =  \frac{n}{\sqrt{\beta_n}} \, ,
\end{equation}
which is precisely the statement of Conjecture \ref{C:Intermediate} for the intermediate regime. 
For general $\beta_n = n^{1 + \gamma}$, with $0 < \gamma < 1/3$, the kernel graph $K$ is still a bounded degree (with degree depending on $\gamma$) expander; however, the degree sequence becomes more complicated \cite{DKLP11}.

\section*{Acknowledgements}
The authors would like to thank an anonymous referee for their careful reading and valuable comments. R.~Sun is supported by NUS Tier 1 grant A-8001448-00-00. 
	M.~Salvi is supported by the MUR Excellence Department Project MatMod@TOV awarded to the Department of Mathematics, University of Rome Tor Vergata, CUP E83C18000100006. He also acknowledges financial support from the MUR 2022 PRIN project GRAFIA, project code 202284Z9E4, and thanks the INdAM group GNAMPA.
\bibliographystyle{plain}
\bibliography{RSTRE}

\end{document}